[1994/12/01]
\documentclass[reqno,12pt]{amsart}
\title{Maximal $L^p$-regularity for perturbed evolution equations in Banach spaces}
\author{A. AMANSAG, H. BOUNIT, A. DRIOUICH, and S. HADD}
\address{Department of Mathematics, Faculty of Sciences, Ibn Zohr University, Hay Dakhla, BP8106, 80000--Agadir, Morocco; ahmed.amansag@uiz.ac.ma, h.bounit@uiz.ac.ma, a.driouich@uiz.ac.ma, s.hadd@uiz.ac.ma}
\thanks{21/10/2018 Version}

\usepackage{amsmath,amsthm}
\usepackage[mathscr]{euscript}
\usepackage{chngcntr}
\usepackage{amsfonts}
\usepackage{amssymb}
\usepackage{latexsym}

\newtheorem{thm}{Theorem}[section]
\newtheorem{prop}[thm]{Proposition}
\newtheorem{lem}[thm]{Lemma}
\newtheorem{cor}[thm]{Corollary}
\theoremstyle{definition}

\newtheorem{dfn}[thm]{Definition}
\theoremstyle{remark}
\newtheorem{rem}[thm]{Remark}

\numberwithin{equation}{section}

\numberwithin{equation}{section}

\newcommand{\sgT}{\mathbb{T}}

\renewcommand{\le}{\leqslant}\renewcommand{\leq}{\leqslant}
\renewcommand{\ge}{\geqslant}\renewcommand{\geq}{\geqslant}

\newcommand{\R}{{\mathbb R}}
\newcommand{\T}{{\mathbb T}}                   
\newcommand{\N}{{\mathbb N}}                  
\newcommand{\C}{{\mathbb C}}                  
\newcommand{\F}{{\mathbb F}}
\newcommand{\D}{{\mathbb D}}
\newcommand{\B}{{\mathbb B}}
\newcommand{\Ss}{{\mathbb S}}
\newcommand{\A}{{\mathbb A}}

\def\ka{\kappa}
\def\al{\alpha}
\def\om{\omega}

\def\ga{\gamma}
\def\si{\sigma}
\def\Si{\Sigma}
\def\la{\lambda}

\def\t{\tau}

\def\calR{{\mathcal{R}}}
\def\calL{{\mathcal{L}}}

\def\calA{{\mathcal{A}}}

\def\calC{{\mathcal{C}}}

\def\R{\mathbb R}

\def\P{\mathbb P}
\def\D{\mathbb D}

\def\C{\mathbb C}

\def\N{\mathbb N}

\def\F{\mathbb F}
\def\K{\mathbb K}

\def\T{\mathbb T}
\def\G{\mathbb G}

\begin{document}
\maketitle

\renewcommand{\sectionmark}[1]{}
\begin{abstract}
The main purpose of this paper is to investigate the concept of maximal $L^p$-regularity for perturbed evolution equations in Banach spaces. We mainly consider three classes of perturbations: Miyadera-Voigt perturbations, Desch-Schappacher perturbations, and more general Staffans-Weiss perturbations. We introduce conditions for which the maximal $L^p$-regularity can be preserved under these kind of perturbations. We give examples for a boundary perturbed heat equation in $L^r$-spaces and a perturbed boundary integro-differential equation. We mention that our results mainly extend those in the works: [P. C. Kunstmann and L. Weis, Ann. Scuola Norm. Sup. Pisa Cl. Sci. (4) 30 (2001), 415-435] and  [B.H. Haak, M. Haase, P.C. Kunstmann,  Adv. Differential Equations 11 (2006), no. 2, 201-240].
\end{abstract}

\section{Introduction}\label{sec:intro}

In this paper we investigate the maximal $L^p$--regularity of evolution equations of the type
\begin{align}\label{PM} \begin{cases} \dot{z}(t)=A_m z(t)+Pz(t)+f(t),& t\ge 0,\cr z(0)=0,\cr Gz(t)=Kz(t),& t\ge 0,\end{cases}\end{align}
where $A_m:Z\subset X\to X$ is a linear closed operator in a Banach space $X$ with domain $D(A_m)=Z,$ $P:Z\to X$ is an additive linear perturbation of $A_m,$  $G,K:Z\to U$ are linear boundary operators  ($U$ is another Banach space) and $f\in L^p(\R^+,X)$ with $p\ge 1$ is a real number. Actually, we assume that $A:=A_m$ with domain $D(A)=\ker(G)$ is a generator of a strongly continuous semigroup $\T:=(\T(t))_{t\ge 0}$ on $X$.

The concept of maximal regularity has been the subject of several works for many years, e.g.  \cite{CanVes,Dore,CoulDuon,Simon}, and the monograph  \cite{DenHiePru}. The main purpose of these works is to give sufficient conditions on the operator $A$ so as the problem \eqref{PM} with $P\equiv 0$ and $K\equiv 0,$ which can be written as  \begin{align}\label{A-only}\dot{z}=Az+f,\;z(0)=0,\end{align} has a maximal $L^p$--regularity. A necessary condition for the  maximal regularity of the evolution equation \eqref{A-only} is that $A$ is a generator of an analytic semigroup. This condition is also sufficient only if the work space $X$ is a Hilbert space, see \cite{Simon}. In the case of UMD Banach space $X,$ Weis \cite{LutzWeis2} introduced necessary and sufficient conditions in terms of $\calR$-boundedness of the operator $A$ for the maximal $L^p$-regularity of \eqref{A-only}.

Let us first consider the problem \eqref{PM} with $K\equiv 0,$ which can be regarded as the following evolution equation
\begin{align}\label{AP-only}\dot{z}=(A+P)z+f,\;z(0)=0.\end{align}
 The maximal $L^p$-regularity for \eqref{AP-only} has been studied by many authors, e.g. \cite{AmaHieSim,Dore,HaHaKu,KunWei,PruSoh}. In \cite{KunWei}, the authors proved that the problem \eqref{AP-only} has the maximal $L^p$--regularity if $X$ is a UMD space, the problem  \eqref{A-only} has a maximal $L^p$--regularity and the perturbation $P$ is small, i.e. for every $\delta>0,$ there exists $c_\delta>0$ such that $\|P x\|\le \delta \|Ax\|+c_\delta \|x\|,$ for all $x\in D(A)$. The proof is mainly based on the fact that $\mathcal{R}$-sectoriality is preserved under $A$-small perturbations and the UMD property of the space $X$. In Section \ref{sec:4} we generalize this result to the case of general Banach spaces. In fact, we will assume that the restriction of $P$ on $D(A)$ is a $p$--admissible observation operator for $A$ (see Section \ref{sec:2} for the definition), hence $(A+P,D(A))$ is a generator of a strongly continuous semigroup on the Banach space $X$, see \cite{Hadd}. In Theorem \ref{theorem:Miyadera} we will show  that if the problem \eqref{A-only} has a maximal $L^p$--regularity then it is also for its perturbed problem \eqref{AP-only}. To compare this result with the result proved in \cite{KunWei}, we first prove in Lemma \ref{lemma:observation_fractional_power} that the fact that $A$ generates an analytic semigroup on a Banach space $X$ and $P$ is $p$--admissible observation operator for $A,$ implies that for $\beta>\frac{1}{p}$ the operator $P(-A)^{-\beta}$ has a bounded extension to $X$. As $(-A)^\beta$ is a $A$--small perturbation and $Px= P(-A)^{-\beta} (-A)^\beta x$ then $P$ is a $A$--small perturbation, and then the result in \cite{KunWei} holds also in Banach spaces.

Return now to our initial boundary problem \eqref{PM}. This later can be reformulated as
\begin{align}\label{calA-max} \dot{z}=(\calA+P)z+f,\qquad z(0)=0,\end{align}
where $\calA: D(\calA)\subset X\to X$ is the linear operator defined by
\begin{align*} \calA:=A_m,\qquad D(\calA)=\{x\in Z: Gx=Kx\}.\end{align*}
In addition to our assumption at the beginning  of this section, we also suppose that $G:Z\to U$ is surjective. Let then $\D_\la\in\calL(U,Z)$ ($\la\in\rho(A)$) be the Dirichlet operator associated with $A_m$ and $G$, see Section \ref{sec:3}. In order to state our main results on well-posedness and maximal $L^p$-regularity of the problem \eqref{calA-max}, we select $B:=(\la I-A_{-1})\D_\la\in\calL(U,X_{-1})$ for $\la\in\rho(A)$, where $X_{-1}$ is the extrapolation space associated to $A$ and $X$, and $A_{-1}:X\to X_{-1}$ is the extension of $A$ to $X,$ which is a generator of a strongly continuous semigroup on $X_{-1}$. We assume that $B$ is a $p$-admissible control operator for $A$, see the next section for the definition and notation. If $K$ is bounded, i.e. $K\in\calL(X,U)$ then it is known that the operator $\calA$ coincides with the part of the operator $A_{-1}+BK$ in $X$, which generates a strongly continuous semigroup on $X$ (see e.g. \cite{{Grei},HaddManzoRhandi,Sala}). In this case, we prove (see Theorem \ref{Desch_SchapThm}) that if the problem \eqref{A-only} has the  maximal $L^p$-regularity, then the evolution equation \begin{align}\label{calA-only}\dot{z}=\calA z+f,\quad z(0)=0\end{align} has also the same property. In addition if we assume that $(A,B,P_{|D(A)})$ generates a regular linear system on $X,U,X$, then $P$ is still $p$-admissible observation operator for $\calA$ and then the problem \eqref{calA-max} is well-posed and has the  maximal $L^p$-regularity, see Theorem \ref{theorem:generalisation_of_Hadd_Manzo_Rhandi} and Theorem \ref{max-reg-DS+MV}. Let us now assume that the boundary operator $K$ is unbounded $K:Z\varsubsetneq X\to U$. This situation is quite difficult which needs additional assumption to treat the well-posedness and  maximal $L^p$-regularity. According to \cite{HaddManzoRhandi}, if we assume that $(A,B,K_{|D(A)})$ is regular on $X,U,U$ with $I_U:U\to U$ as an admissible feedback, then the problem \eqref{calA-only} is well-posed on the Banach space $X$. Moreover, if the problem \eqref{A-only} has a  maximal $L^p$-regularity and $\|\la \D_\la\|\le \kappa$ for any ${\rm Re}\la>\la_0,$ where $\la_0\in\R$ and $\kappa>0$ are constants, then the problem \eqref{calA-only} has also the maximal $L^p$-regularity on a non reflexive Banach space $X$, see Theorem \ref{theorem:Staffans_Weiss_Non_Reflexive_Case}. On the other hand, we assume that  $(A,B,P_{|D(A)})$ generates a regular linear system on $X,U,X$. Then, in Theorem \ref{theorem:generalisation_of_Hadd_Manzo_Rhandi}, we prove that the problem \eqref{calA-max} is well-posed. Corollary \ref{SF-big-equa-result} shows that the problem \eqref{calA-max} has the  maximal $L^p$-regularity. If $X$ is a UMD space then we use $\calR$-boundedness to prove the  maximal $L^p$-regularity for the evolution equation \eqref{calA-max}, see Theorem \ref{theorem:staffans-weiss-R-boundedness} and Corollary \ref{SF-big-equa-result}. We mention that in \cite{HaHaKu}, the authors proved perturbation theorems for sectoriality and $\calR$-sectoriality in general Banach spaces. They gives conditions on intermediate spaces $Z$ and $W$ such that, for an operator $S:Z \rightarrow W$ of small norm, the operator $A+S$ is sectorial (resp. $\calR$-sectorial) provided $A$ is sectorial (resp. $\calR$-sectorial). Their results are obtained by factorizing $S=BC$. As $\calR$-sectoriality implies maximal regularity in UMD spaces, these theorems yield to maximal regularity perturbation only in UMD spaces.

In Section \ref{sec:5},  we have used product spaces and Bergman spaces to reformulate boundary perturbed intego-differential equations as our abstract boundary evolution equation \eqref{PM}. This allows us to translate the results on well-posedness and maximal $L^p$--regularity obtained for the problem \eqref{PM} to intego-differential equations.

 In the next section, we first recall the necessary material about feedback theory of infinite dimensional linear systems. We then use this theory to prove the well-posedness of the evolution equation \eqref{PM} in Section \ref{sec:3}. Our main results on maximal $L^p$--regularity for the problem \eqref{PM} are gathered in Section \ref{sec:4}. The last section is devoted to apply the obtained results to perturbed intego-differential equations.

 {\bf Notation.} Hereafter $p,q\in [1,\infty]$ and $T>0$ are real numbers such that $\frac{1}{p}+\frac{1}{q}=1$. If $X$ is a Banach space, we denote by $L^p([0,T];X)$ the space of all $X$-valued Bochner integrable functions. For any $\theta \in (0,\pi)$,  $\Sigma_\theta$ is the following sector:
$$
\Sigma_\theta:=\{z\in \mathbb{C}^*; \vert arg z\vert < \theta \}.
$$
For any $\alpha \in \mathbb{R}$, the right half-plane is defined by
$$
\mathbb{C}_\alpha=:\{z\in \mathbb{C}; Re z>\alpha\}.
$$
Given a semigroup $\sgT:=(\sgT(t))_{t\geq 0}$ generated by an operator $A:D(A)\subset X\to X$, we will always denote by $\omega_0(\mathbb{T})$(or $\omega_0(A)$) the growth bound of this semigroup. The resolvent set of $A$ is denoted by $\rho(A)$. Preferably, we denote the resolvent operator of $A$ by $R(\la,A):=(\la-A)^{-1}$ for any $\la\in\rho(A),$ where the notation $\la-A$ means $\la I-A$.

\section{Feedback theory of infinite dimensional linear systems}\label{sec:2}
In this section, we gather definitions and results from feedback theory of infinite dimensional linear systems mainly developed in the references \cite{Sala,Staff,TucWei,WeiRegu}.  We also give some new development of this theory. Hereafter, $X$ and $U$ are Banach spaces and $p\in [1,\infty[$.

It is known (see e.g. \cite{Sala,Staff}) that partial differential equations with boundary control and point observation can be reformulated as the following distributed linear system
\begin{equation}
  \label{linContSys}
  \begin{cases}
  \dot{x}(t) = Ax(t)+Bu(t) , & \mbox{ } t\geq 0 \\
  y(t) = Kx(t), & \mbox{}\\
  x(0)=x_0, & \mbox{}
  \
  \end{cases}
\end{equation}
where $A:D(A)\subset Z\subset X\to X$ is the generator of a strongly continuous semigroup $\T:=(\T(t))_{t\ge 0}$ on $X$ with $Z$ is a Banach space continuously and densely embedded in $X$, $B\in\calL(U,X_{-1})$ is a control operator such that
\begin{align*}
R(\la,A_{-1})B\in\calL(U,Z),\qquad \la\in\rho(A),
\end{align*}
and $K\in\calL(Z,U)$ is an observation operator. Here  $X_{-1}$ is the completion of $X$ with respect to the norm $\Vert R(\lambda,A)\cdot\Vert$. We recall that we can extend $\T$ to another strongly continuous semigroup $\T_{-1}:=(\sgT_{-1}(t))_{t\geq 0}$ on $X_{-1}$ with  generator $A_{-1}:X\to X_{-1}$, the extension of $A$ to $X$ (see e.g. \cite[chap.2]{EngNag}).
The mild solution of the system \eqref{linContSys} is given by:
\begin{equation}
\label{solVarConstForm}
x(t) = \sgT(t)x_0 + \int_0^t \sgT_{-1}(t-s)Bu(s) ds \qquad x_0 \in X,
\end{equation}
where the integral is taken in $X_{-1}$. Formally, the well-posedness of the system  \eqref{linContSys} means that the state satisfies $x(t)\in X$ for any $t\ge 0,$ the observation function $y$ is extended to a locally $p$-integrable function  $y\in L^p_{loc}([0,\infty),U)$ satisfying the following property:  for any $\t>0,$ there exists a constant $c_\tau>0$ such that
\begin{align*}
\|y\|_{L^p([0,\t],U)}\le c_\t \left(\|x_0\|+\|u\|_{L^p([0,\t],U)}\right),
\end{align*}
for any initial state $x_0\in X$ and any control function $u\in L^p_{loc}([0,\infty),U)$. In order to mathematically explain this concept, let us define
\begin{align*}
C:=K_{|D(A)}\in\calL(D(A),U).
\end{align*}
We also need the following definition.
\begin{dfn}
\begin{itemize}
\item[(i)]  $B\in \mathcal{L}(U,X_{-1})$ is called $p$-admissible control operator for $A$, if there exists $t_0 > 0$ such that :
\begin{equation*}
\Phi_{t_0} u := \int_0^{t_0} \sgT_{-1}(t-s)Bu(s) ds \in X
\end{equation*}
for any $u\in L^p_{loc}([0,\infty),U)$. We also say that $(A,B)$ is $p$-admissible.
\item[(ii)] $C\in \mathcal{L}(D(A),Y)$ is called $p$-admissible observation operator for $A$, if there exist $\alpha >0$ and $\kappa:=\kappa_\al>0$ such that:
\begin{equation}\label{walid}
 \int_0^\alpha \Vert C\sgT(t)x\Vert_Y^p dt \leq \ka^p \Vert x \Vert^p,
\end{equation}
for all $x\in D(A)$. We also say that $(C,A)$ is $p$-admissible.
\end{itemize}
\end{dfn}
Let us now describe some consequences of this definition. If $B$ is $p$-admissible control operator for $A,$ then by the closed graph theorem one can see that for any $t\ge 0,$ $$\Phi_t\in\calL(L^p([0,t],U),X).$$
This implies that the state of the system  \eqref{linContSys} satisfies $x(t)=\T(t)x_0+\Phi_t u\in X$ for any $t\ge 0,\;x_0\in X$ and $u\in L^p_{loc}([0,\infty),U)$.  According to \cite{Weis-ad-cont}, for all $0<\t_1\le \t_2,$
\begin{align}\label{norm-phi-estim}
\|\Phi_{\t_1}\|\le \|\Phi_{\t_2}\|.
\end{align} Now if $C$ is $p$-admissible observation operator for $A,$ then due to \eqref{walid}, the map $\Psi_\infty: D(A)\to L^p_{loc}([0,\infty),U)$ defined by $\Psi_\infty x := C\sgT(\cdot)x$, can be  extended to a bounded operator $\Psi_\infty: X\rightarrow L_{loc}^p([0,\infty);U)$. For any $x\in X$ and $t\ge 0$, we define the family $\Psi_t x := \Psi_\infty x$ on $[0,t]$. Then for all $t\ge 0,$ $$\Psi_t\in \mathcal{L}\left( X, L^p_{loc}([0,\infty),U)\right).$$
On the other hand, let us consider the linear operator
\begin{align*}
D(C_\Lambda)&:=\left\{x\in X: \lim_{s \rightarrow +\infty} s C R(s,A)x \;\text{exists in}\; U\right\}\\
C_\Lambda x &:= \lim_{s \rightarrow +\infty}s C R(s,A)x.
\end{align*}
Clearly, $D(A)\subset D(C_\Lambda)$ and $C_\Lambda=C$ on $D(A)$. This shows that $C_\Lambda$ is in fact an extension of $C,$ called the \textit{Yosida} extension of $C$ w.r.t. $A$. We note that if $C$ is $p$-admissible for $A,$ then $\T(t)X\subset D(C_\Lambda)$  and
\begin{align*}
(\Psi_\infty x)(t)=C_\Lambda \T(t)x,
\end{align*}
for any $x\in X$ and a.e. $t>0$.

In the sequel, we assume that $B$ and $C$ are $p$-admissible for $A$ and set
\begin{align*}
W^{2,p}_{0,loc}([0,\infty),U):=\left\{u\in W^{2,p}_{loc}([0,\infty),U):u(0)=0\right\}.
\end{align*}
This space is dense in $L^p_{loc}([0,\infty),U)$. Remark that for any $u\in W^{2,p}_{0,loc}([0,\infty),U),\; t\ge 0$ and by assuming $0\in \rho(A)$ (without loss of generality) and using an integration by parts, we have
\begin{align*}
\Phi_t u=R(0,A_{-1})Bu(t)-R(0,A)\Phi_t \dot{u}\in Z.
\end{align*}
On the other hand, using the fact that $KR(0,A_{-1})B\in\calL(U),\;CR(0,A)\in\calL(X,U)$ and \eqref{norm-phi-estim}, the application $(t\mapsto K\Phi_t u)\in L^p_{loc}([0,\infty),U)$ for any $u\in W^{2,p}_{0,loc}([0,\infty),U)$. Thus we have defined an application
\begin{align}\label{inp-out-infini}
\F_\infty: W^{2,p}_{0,loc}([0,\infty),U)\to L^p_{loc}([0,\infty),U),\qquad u\mapsto \F_\infty u= K\Phi_{\cdot} u.
\end{align}
\begin{dfn}\cite{WeiAdmiss}
Let $B$ and $C$ be $p$-admissible control and observation operators for $A,$ respectively. We say that the triple $(A,B,C)$ generates a well-posed system $\Sigma$ on $X,U,U$, if the operator $\F_\infty$ defined by \eqref{inp-out-infini} satisfies the following property: For any $\al>0$ there exists a constant $\vartheta_\al>0$ such that for all $u\in  W^{2,p}_{0,loc}([0,\infty),U),$
\begin{align}\label{F-infty-estim}
\left\| \F_\infty u\right\|_{L^p([0,\al],U)}\le \vartheta_\al \|u\|_{L^p([0,\al],U)}.
\end{align}
The operator $\F_\infty$ is called the extended input-output operator of $\Si$.
\end{dfn}
If $(A,B,C)$ generates a well-posed system $\Sigma$ on $X,U,U$, then we have two folds: first the state of \eqref{linContSys} satisfies $x(t)\in X$ for all $t\ge 0,$ and second $\F_\infty$ have an extension $\F_\infty\in \calL(L^p_{loc}([0,\infty),U)),$ due to \eqref{F-infty-estim}. Observe that the observation function $y$ verifies
\begin{align}\label{y-localy}
y(\cdot):=y(\cdot;x_0,u)=CT(\cdot)x_0+\F_\infty u=(\begin{smallmatrix}\Psi_\infty& \F_\infty\end{smallmatrix})(\begin{smallmatrix}x_0\\ u\end{smallmatrix}),
\end{align}
for all $x_0\in D(A)$ and $u\in  W^{2,p}_{0,loc}([0,\infty),U)$. By   density of $D(A)\times W^{2,p}_{0,loc}([0,\infty),U)$ in $X\times L^p_{loc}([0,\infty),U)$, the function $y$ is extended to a function $y\in L^p_{loc}([0,\infty),U)$ such that
\begin{align*}
y=\Psi_\infty x_0+\F_\infty u,\qquad \forall (x_0,u)\in X\times L^p_{loc}([0,\infty),U).
\end{align*}
We now turn out to give a representation of the observation function $y$ in terms of the observation operator $C$ and the state $x(\cdot)$. To that purpose Weiss \cite{WeiRegu,WeiTrans} introduced the following subclass of well-posed linear systems.
\begin{dfn}
Let $(A,B,C)$ generates a well-posed system $\Si$ on $X,U,U$ with   extended input-output operator $\F_\infty$. This system is called regular (with feedthrough $D=0$) if :
$$
\lim_{\tau \rightarrow 0} \frac{1}{\tau} \int_0^\tau (\mathbb{F}_\infty u_{z_0})(s)ds = 0
$$
with $u_{z_0}(s) = z_0$ for all $s\geq 0 $, is a constant control function.
\end{dfn}
According to Weiss \cite{WeiRegu,WeiTrans}, if $(A,B,C)$ generates a regular system $\Si$ on $X,U,U$, then the state and the observation function of the linear system \eqref{linContSys} satisfy
\begin{align}\label{representation-y}
x(t)\in D(C_\Lambda)\quad\text{and}\quad y(t)=C_\Lambda x(t),
\end{align}
for any initial state $x(0)=x_0\in X$, any control function $u\in L^p([0,\infty),U)$ and a.e. $t\ge 0$.
\begin{dfn}\label{Feedback-admisible}
Let a triple  $(A,B,C)$ generates a well-posed system $\Si$ on $X,U,U$ with   extended input-output operator $\F_\infty$. Define
$$
 \mathbb{F}_\tau u:= \mathbb{F}_\infty u \,\qquad\text{on}\quad [0,\tau].
$$
The identity operator $I_U:U\to U$ is called an admissible feedback for $\Sigma$ if the operator $I-\mathbb{F}_{t_0}:L^p([0,t_0],U)\to L^p([0,t_0],U)$ admits a (uniformly) bounded inverse for some $t_0>0$ (hence all $t_0>0$).
\end{dfn}
A consequence of Definition \ref{Feedback-admisible} is that the feedback law $u=y(\cdot;x_0,u)$ has a sense. In fact, due to \eqref{y-localy} this is equivalent to $(I-\F_{\t})u=\Psi_\t x_0$ on $[0,\t]$. As $I-\F_{\t}$ is invertible in $L^p([0,\t],U),$ then the equation $u=y(\cdot;x_0,u)$ has a unique solution and this solution $u\in L^p([0,\t],U)$ is given also by
\begin{align*}
u(t)=C_\Lambda x(t),\qquad a.e.\; t\ge 0,
\end{align*}
due to \eqref{representation-y}. Using \eqref{solVarConstForm}, the state $x(\cdot)$ satisfies the following variation of constants formula
\begin{align*}
x(t)=\T(t)x_0+\int^t_0 \T_{-1}(t-s)BC_\Lambda x(s)ds
\end{align*}
for any $x_0\in X$ and any $t\ge 0$. Now we set
\begin{align*}
\T^{cl}(t)x_0:=x(t),\qquad t\ge 0.
\end{align*}
Then by using the definition of $C_0$--semigroups one can see that $(\T^{cl}(t))_{t\ge 0}$ is a $C_0$--semigroup on $X$. More precisely, we have the following perturbation theorem due to Weiss \cite{WeiRegu} in Hilbert spaces and to Staffans \cite[Chap.7]{Staff} in Banach spaces.
\begin{thm}\label{staff_wei}
Let $(A,B,C)$ generates a regular linear system $\Sigma$ with the identity operator  $I_U:U\to U$ an admissible feedback operator. Then the operator
\begin{align} \label{A-cl}
\begin{split}
A^{cl} &:= A_{-1} + BC_\Lambda\\
D(A^{cl}) &:= \{x\in D(C_\Lambda); (A_{-1} + BC_\Lambda)x\in X\}
\end{split}
\end{align}
generates a $C_0$-semigroup $(\sgT^{cl}(t))_{t\geq 0}$ on $X$ such that $range(\sgT^{cl}(t))\subset D(C_\Lambda)$ for a.e. $t>0$, and  for any $\al>0,$ there exists $c_\al>0$ such that for all $x_0\in X,$
\begin{align}\label{T-cl-estim}
\|C_\Lambda \T^{cl}(\cdot)x_0\|_{L^p([0,\al],U)}\le c_\al \|x_0\|.
\end{align}
Moreover, this semigroup satisfies
\begin{equation}\label{WS-VCF}
\sgT^{cl}(t)x_0 = \sgT(t)x_0 + \int_0^t \sgT_{-1}(t-s)BC_{\Lambda}\sgT^{cl}(s)x_0 ds \qquad \qquad x_0\in X\ ,\ t \geq 0.
\end{equation}
In addition  $(A^{cl},B,C_\Lambda)$ generates a regular system $\Sigma^{cl}$.
\end{thm}
%
\begin{dfn}\label{Staffans-Weiss-perturbation-oper}
Let $(A,B,C)$ generates a regular linear system on $X,U,U$ with the identity operator  $I_U:U\to U$ as an admissible feedback. The operator $$\P^{sw}:=BC:D(C_\Lambda)\subset X\to X_{-1}$$ is called the {\bf Staffans-Weiss perturbation} of $A$.
\end{dfn}

It is not difficult to see that if one of the operators $B$ or $C$ is bounded (i.e. $B\in\calL(U,X)$ or $C\in\calL(X,U)$) and the other is $p$-admissible then the triple $(A,B,C)$ generates a regular linear system on $X,U,U$ with the identity operator $I_U:U\to U$ as an admissible feedback. As application of the Staffans-Weiss theorem (Theorem \ref{staff_wei}), we distinct two subclasses of perturbations as follows:
\begin{rem}\label{Hadd_thm}
\begin{itemize}
  \item [{\rm (i)}] We take $B\in \calL(X,U)$ and $C\in\calL(D(A),U)$ a $p$-admissible observation operator for $A$. According to Theorem \ref{staff_wei}, the operator $A^{cl}:=A+BC$ with domain $D(A^{cl})=D(A)$ is a generator of a strongly continuous semigroup $\T^{cl}:=(\T^{cl}(t))_{t\ge 0}$ on $X$ such that $\T^{cl}(t)X\subset D(C_\Lambda)$ for a.e. $t>0,$ the estimate \eqref{T-cl-estim} holds, and
      \begin{align}\label{MV-VCF1}
      \T^{cl}(t)x=\T(t)x+\int^t_0 \T(t-s)BC_\Lambda \T^{cl}(s)xds,\qquad t\ge 0,\; x\in X.
      \end{align}
      On the other hand, it is shown in \cite{Hadd}, that the semigroup $\T^{cl}$ satisfy also the following formula
      \begin{align}\label{MV-VCF2}
      \T^{cl}(t)x=\T(t)x+\int^t_0 \T^{cl}(t-s)BC_\Lambda \T(s)xds,\qquad t\ge 0,\; x\in X.
      \end{align}
      Using H\"{o}lder inequality on can see that there exists $\al_0>0$ and $\ga\in (0,1)$ such that
      \begin{align*}
      \int^{\al_0}_0 \|BC\T(t)x\|\le \ga \|x\|
      \end{align*}
      for all $x\in D(A)$. The following operator
      $$\P^{mv}:=BC: D(A)\to X $$
      is a {\bf Miyadera-Voigt perturbation} for $A;$ (see e.g. \cite[p.195]{EngNag}).
  \item [{\rm (ii)}] We take $C\in\calL(X,U)$ and $B\in\calL(U,X_{-1})$ a $p$-admissible control operator for $A$. Then the part of the operator $A_{-1}+BC$ in $X$ generates a strongly continuous semigroup on $X$ satisfying all properties of Theorem \ref{staff_wei}. In this case the operator $$\P^{ds}:=BC:X\to X_{-1}$$ is called {\bf Desch-Schappacher perturbation} for $A$ (see e.g. \cite[p.182]{EngNag}).
\end{itemize}
\end{rem}

\section{Well-posedness of perturbed boundary value problems}\label{sec:3}
The object of this section is to investigate the well-posedness of the perturbed boundary value problem defined by \eqref{PM}. We first rewrite \eqref{PM} as non-homogeneous perturbed Cauchy problem of the form \eqref{calA-max}. Then the  well-posedness of \eqref{PM} can be obtained if for example the operator
\begin{align}\label{cald-def} \calA:=A_m,\qquad D(\calA)=\{x\in Z: Gx=Kx\}\end{align}  generates a strongly continuous semigroup on $X$ and that $P$ is a $p$-admissible observation operator for $A$ (see Remark \ref{Hadd_thm} (i)). Recently, the authors of \cite{HaddManzoRhandi} introduced conditions on $A_m,G$ and $K$ for which $\calA$ is a generator. To be more precise, assume that
\begin{itemize}
\item[] \textbf{(H1)} $G:Z\to U$ is onto, and
\item[] \textbf{(H2)} the operator defined by $A:=A_{m|\ker(G)}$ and $D(A):=\ker(G)$, generates a $C_0$-semigroup $(\sgT(t))_{t\geq 0}$.
\end{itemize}
According to  Greiner \cite{Grei}, these conditions imply that for any $\la\in\rho(A)$ the restriction of $G$ to $\ker(\la-A_m)$ is invertible. We then define
\begin{align*}
\D_\la:=(G _{|Ker(\lambda-A_m}))^{-1}\in\calL(U,X),\qquad \la\in\rho(A).
\end{align*}
This operator is called the {\em Dirichlet operator}.  Define the operators :
\begin{align}\label{B_op}
\begin{split}
  B:= & (\lambda - A_{-1})\mathbb{D}_\lambda \in \mathcal{L}(U,X_{-1}), \\
   C:=  & Ki \in \mathcal{L}(D(A),U),
\end{split}
\end{align}
where $i$ is the canonical injection from $D(A)$ to $Z$. In the rest of this paper, $C_\Lambda$ denotes the Yosida extension of $C$ with respect to $A$.
It is shown in \cite[lem.3.6]{HaddManzoRhandi} that if $A,B,C$ as above and if $(A,B,C)$ generates a regular linear system $\Si$ on $X,U,U,$ then we have
\begin{align}\label{relation-K-C-Lambda}
Z\subset D(C_\Lambda)\quad\text{and}\quad (C_\Lambda)_{|Z}=K.
\end{align}
If $H$ is the transfer function of $\Si$ and $\al>\om_0(A)$ then
\begin{align}\label{transfer-function-ABC}
H(\la)=C_\Lambda R(\la,A_{-1})B=C_\Lambda \D_\la=K\D_\lambda,
\end{align}
for any $\la\in\C$ with ${\rm Re}\la >\al$. Moreover, we have
\begin{align}\label{limit-trans-func}
\lim_{s\to+\infty}H(s)=0.
\end{align}
We have the following perturbation theorem (see \cite{HaddManzoRhandi} for the proof).
\begin{thm}\label{theorem:Hadd_Manzo_Rhandi}
Let assumptions {\rm \textbf{(H1)}} and {\rm \textbf{(H2)}} be satisfied and let $B$ and $C$ be the operators defined in \eqref{B_op}. Assume that $(A,B,C)$ generates a regular linear system on $X,U,U$ with the identity operator $I_U:U\to U$ as an admissible feedback. The following assertions hold:
\begin{itemize}
  \item [{\rm (i)}]The operator $(\calA,D(\calA))$ defined by \eqref{cald-def} coincides with the following operator
  \begin{align*}
  A^{cl}:=A_{-1}+BC_\Lambda,\quad D(A^{cl})=\{x\in D(C_\Lambda):(A_{-1}+BC_\Lambda)x\in X\}.
  \end{align*}
  \item [{\rm (ii)}] The operator $(\calA,D(\calA))$ generates a strongly continuous semigroup $(\sgT^{cl}(t))_{t\geq 0}$ on $X$ as in Theorem \ref{staff_wei}.
  \item [{\rm (iii)}] For any $\lambda\in\rho(A)$ we have
\begin{equation*}
\lambda\in\rho(\calA) \Leftrightarrow 1 \in \rho(\mathbb{D}_\lambda K) \Leftrightarrow 1 \in \rho(K\mathbb{D}_\lambda).
\end{equation*}
\item [{\rm (iv)}] Finally for $\lambda \in \rho(A)\cap \rho(\calA)$:
\begin{equation*}
R(\lambda,\calA) = (I-\mathbb{D}_\lambda K)^{-1}R(\lambda,A).
\end{equation*}
\end{itemize}
\end{thm}
Under the assumptions of Theorem \ref{theorem:Hadd_Manzo_Rhandi}, the mild solution of the problem \eqref{calA-only} is given by
\begin{align}\label{formula-used}
z(t)=\T^{cl}(t)x+\int^t_0 \T^{cl}(t-s)f(s)ds,
\end{align}
for any $t\ge 0,\;x\in X$ and $f\in L^p(\R^+,X)$. Before giving  another useful expression of $z$ in term of the semigroup $\T$, we need the following very useful result proved in \cite[prop.3.3]{Hadd}.
\begin{lem}\label{Hadd_lemma}
let $(\Ss(t))_{t\ge 0}$ be a strongly continuous semigroup on $X$ with generator $(\mathbb{G},D(\mathbb{G}))$. Let $\Upsilon\in\calL(D(\mathbb{G}),X)$ be a $p$-admissible observation operator for $\mathbb{G}$. Denote by $\Upsilon_\Lambda$ the Yosida extension of $\Upsilon$ with respect to $\mathbb{G}$. Then
\begin{align*}
&(\Ss\ast f)(t):=\int^t_0 \Ss(t-s)f(s)ds\in D(\Upsilon_\Lambda),\quad \text{a.e.}\;t\ge 0,\cr
&\left\|\Upsilon_\Lambda (\Ss\ast f)\right\|_{L^p([0,\alpha],X)} \leq c(\alpha)\|f \|_{L^p([0,\alpha],X)},
\end{align*}
for $\alpha >0$, $f\in L_{loc}^p([0,\infty),X)$ and a constant $c(\alpha)$ independent of $f$ such that $c(\alpha)\rightarrow 0$ as $\alpha\rightarrow 0$.
\end{lem}

\begin{prop}\label{VCF-calA-f}
Let assumptions of Theorem \ref{theorem:Hadd_Manzo_Rhandi} be satisfied. Let $\al>0$ and $f\in L^p(\R^+,X)$.  The non-homogenous Cauchy problem \eqref{calA-only} is well-posed and its mild solution satisfies for any initial condition $x\in X,$
\begin{align}\label{3-calA-f-z}
\begin{split}
&z(t)\in D(C_\Lambda)\quad\text{a.e.}\; t>0,\cr & \|C_\Lambda z(\cdot)\|_{L^p([0,\al],X)}\le c_\al  \left(\|x\|+\|f\|_{L^p([0,\al],X)}\right),\cr & z(t)=\T(t)x+\int^t_0 \T_{-1}(t-s)BC_\Lambda z(s)ds+\int^t_0 \T(t-s)f(s)ds,\qquad t\ge 0,
\end{split}
\end{align}
where $c_\al>0$ is a constant independent of $f$.
\end{prop}
\begin{proof}
Let, by Theorem  \ref{theorem:Hadd_Manzo_Rhandi}, $\T^{cl}$ the semigroup generated by $\calA$ and let $z:[0,+\infty)\to X$ be the mild solution of the problem \eqref{calA-only} given by \eqref{formula-used}. According to Theorem \ref{staff_wei}, we know that $C_\Lambda$ is an admissible observation operator for $\calA$. We denote by $C_{\Lambda,\calA}$ the Yosida extension of $C_\Lambda$ with respect to $\calA$. Then $D(C_{\Lambda,\calA}) \subset D(C_\Lambda)$ and $C_{\Lambda,\calA}=C_\Lambda$ on $D(C_{\Lambda,\calA})$. In fact, let $x\in D(C_{\Lambda,\calA})$ and $s>0$ sufficiently large. Then by first taking Laplace transform on both sides of \eqref{WS-VCF} and second applying $sC_\Lambda,$ we obtain
\begin{align}\label{CR-formula}
sC_\Lambda R(s,\calA)x=sCR(s,A)+H(s)C_\Lambda s R(s,\calA)x,
\end{align}
where we have used \eqref{transfer-function-ABC}. Remark that
\begin{align*}
\|H(s)C_\Lambda s R(s,\calA)x\|\le \|H(s)\|\left(\|C_\Lambda s R(s,\calA)x-C_{\Lambda,\calA}x\|+ \|C_{\Lambda,\calA}x\|\right)
\end{align*}
Hence, by \eqref{limit-trans-func} and the fact that $x\in D(C_{\Lambda,\calA})$, we obtain
\begin{align*}
\lim_{s\to+\infty}H(s)C_\Lambda s R(s,\calA)x=0.
\end{align*}
Now from \eqref{CR-formula}, we deduce that $x\in D(C_\Lambda)$ and $C_{\Lambda,\calA}x=C_\Lambda x$. Let $x\in X,\;\al>0$ and $f\in L^p([0,\al],X)$. The fact that $C_\Lambda$ is $p$--admissible for $\calA,$ then by using \eqref{formula-used} and Lemma \ref{Hadd_lemma}, we obtain $z(t)\in D(C_{\Lambda,\calA})$ for  a.e. $t>0$.  This shows that $z(t)\in D(C_\Lambda)$ and $C_\Lambda z(t)=C_{\Lambda,\calA}z(t)$ for a.e. $t>0$. The estimation in \eqref{3-calA-f-z} follows immediately from \eqref{T-cl-estim} and Lemma \ref{Hadd_lemma}. Let us prove the last property in \eqref{3-calA-f-z}. By density there exists $(f_n)_n\subset \calC([0,\al],D(\calA))$ such that $f_n\to f$ in $L^p([0,\al],X)$ as $n\to \infty$. We set
\begin{align}\label{expr-z-n}
z_n(t)=\T^{cl}(t)x+\int^t_0 \T^{cl}(t-s)f_n(s)ds,\qquad t\ge 0.
\end{align}
Using H\"{o}lder inequality, it is clear that $\|z_n(t)-z(t)\|\to 0$ as $n\to\infty$. Now let us prove that $z_n$ satisfies the third assertion in \eqref{3-calA-f-z}. In fact, the estimate in \eqref{3-calA-f-z} implies that
\begin{align*}
\|C_\Lambda z_n(\cdot)-C_\Lambda z(\cdot)\|_{L^p([0,\al],X)}\le c_\al \|f_n-f\|_{L^p([0,\al],X)}\underset{n\to\infty}{\longrightarrow}0.
\end{align*}
On the other hand, using the expression of the semigroup $\T^{cl}$ given in \eqref{WS-VCF}, change of variable and Fubini theorem we obtain
\begin{align}\label{z_n}
z_n(t)&=\T^{cl}(t)x+\int^t_0\T(t-s)f_n(s)ds+\int^t_0 \T_{-1}(t-\t)B\int^{\t}_0 C_\Lambda \T^{cl}(\t-s)f_n(s)ds d\t \cr &=\T(t)x+ +\int^t_0\T(t-s)f_n(s)ds\cr & \hspace{1.5cm}+\int^t_0 \T_{-1}(t-\t)B\left(C_\Lambda \T^{cl}(\t)x+\int^{\t}_0 C_\Lambda \T^{cl}(\t-s)f_n(s) ds\right) d\t.
\end{align}
For simplicity we assume that $0\in\rho(\calA)$. We then have
\begin{align*}
C_\Lambda \int^{\t}_0 \T^{cl}(\t-s)f_n(s)ds&=C_\Lambda (-\calA)^{-1} (-\calA)\int^{\t}_0 \T^{cl}(\t-s)f_n(s)ds \cr &=C_\Lambda (-\calA)^{-1} \int^{\t}_0 \T^{cl}(\t-s)(-\calA)f_n(s)ds \cr &=\int^{\t}_0 C_\Lambda (-\calA)^{-1} \T^{cl}(\t-s)(-\calA)f_n(s)ds\cr &=\int^{\t}_0 C_\Lambda \T^{cl}(\t-s)f_n(s)ds.
\end{align*}
Now replacing this in \eqref{z_n}, and using \eqref{expr-z-n}, we have
\begin{align*}
z_n(t)=\T(t)x+\int^t_0 \T_{-1}(t-s)BC_\Lambda z(s)ds+\int^t_0 \T(t-s)f_n(s)ds,\qquad t\ge 0.
\end{align*}
Put
\begin{align*}
\varphi(t)=\T(t)x+\int^t_0 \T_{-1}(t-s)BC_\Lambda z(s)ds+\int^t_0 \T(t-s)f(s)ds,\qquad t\ge 0.
\end{align*}
Then for any $t\in [0,\al],$ we have
\begin{align*}
\|z_n(t)-\varphi(t)\|\le \ga_\al\left( \|C_\Lambda z_n(\cdot)-C_\Lambda z(\cdot)\|_{L^p([0,\al],X)}+ \|f_n-f\|_{L^p([0,\al],X)}\right),
\end{align*}
due to the admissibility of $B$ for $A$ and  H\"{o}lder inequality. This shows that $\|z_n(t)-\varphi(t)\|\to 0$ as $n\to \infty$, and hence $z=\varphi$.
\end{proof}
Now we can state the main result of this section.
\begin{thm}\label{theorem:generalisation_of_Hadd_Manzo_Rhandi}
Let assumptions of Theorem \ref{theorem:Hadd_Manzo_Rhandi} be satisfied. In addition, let $P:Z\to X$ such that  $(A,B,\P)$  generates a regular linear system on $X,U,X$, where $\P=P_{|D(A)}$. The following assertions hold:
\begin{itemize}
  \item [{\rm (i)}] The operator $P\in\calL(D(\calA),X)$ is a $p$-admissible observation operator for $\calA$, hence the operator $(\calA+P,D(\calA))$ generates a strongly continuous semigroup on $X$.
  \item [{\rm (ii)}] The boundary problem \eqref{PM} is well-posed and has a mild solution $z:[0,+\infty)\to X$ satisfying:
  \begin{align*}
 & z(t)\in D(P_\Lambda)\quad\text{a.e.}\; t\ge 0,\cr & z(t)=\T^{cl}(t)x+\int^t_0 \T^{cl}(t-s)\left(P_\Lambda z(s)+f(s)\right)ds
  \end{align*}
  for any $t\ge 0,$ initial condition $x\in X$ and $f\in L^p([0,\infty),X),$ where $P_\Lambda$ is the Yosida extension of $P$ w.r.t $\calA$.
\end{itemize}
\end{thm}
\begin{proof}
(i) We first remark from \eqref{relation-K-C-Lambda} that $Z\subset D(\P_{0,\Lambda})$ and $P=\P_{0,\Lambda}$ on $Z,$ where  $\P_{0,\Lambda}$ denotes the Yosida extension of $\P$ w.r.t. $A$. Let $x\in D(\calA)$ and $\al>0$. The facts that $(A,B,\P)$ is regular and \eqref{T-cl-estim}, we have
\begin{align}\label{hhhhh}
&\int^t_0 \T_{-1}(t-s)BC_\Lambda \T^{cl}(s)x\in D(\P_{0,\Lambda})\quad\text{a.e.}\;t\ge 0,\;\text{and}\cr & \left\| \P_{0,\Lambda}\int^{\cdot}_0 \T_{-1}(t-s)BC_\Lambda \T^{cl}(s)x\right\|_{L^p([0,\al],X)} \le \beta_\al \|x\|,
\end{align}
where $\beta_\al>0$ is a constant. On the other hand, by \eqref{WS-VCF}, we have
\begin{align*}
P\T^{cl}(t)x&=\P_{0,\Lambda}\T^{cl}(t)x\cr &=\P_{0,\Lambda}\T(t)x+\P_{0,\Lambda}\int^{t}_0 \T_{-1}(t-s)BC_\Lambda \T^{cl}(s)x.
\end{align*}
Hence the $p$-admissibility of $P$ for $\calA$ follows by \eqref{hhhhh} and the $p$-admissibility of $\P$ for $A$. Thus, according to Remark \ref{Hadd_thm} (i), the operator $(\calA+P,D(\calA))$ generates a strongly continuous semigroup on $X$.  The assertion (ii) follows from \cite[thm.5.1]{Hadd}
\end{proof}


\section{Perturbation Theorems for maximal regularity}\label{sec:4}

\subsection{Maximal regularity}\label{sub:4.0}

Let $\G:D(\G)\subset X\to X$ be the generator of a strongly continuous semigroup $\Ss:=(\Ss(t))_{t\ge 0}$ on a Banach space $X$. Consider the following non--homogeneous abstract Cauchy problem:
\begin{equation}
  \label{ACP}
  \begin{cases}
  \dot{z}(t) = \G z(t)+f(t) , & \mbox{ } 0<t \leq T \\
  z(0) = 0, & \mbox{}
  \
  \end{cases}
\end{equation} where $f: [0,T]\rightarrow X$ a measurable function.
\begin{dfn}
We say that the operator $\G$ (or the problem \eqref{ACP}) has the maximal $L^p$-regularity on the interval $[0,T]$, and we write $\G\in \mathscr{MR}_p(0,T;X)$, if for all $f\in L^p([0,T],X)$, there exists a unique $z \in W^{1,p}([0,T],X)\cap L^p([0,T],D(\G))$ which verifies (\ref{ACP}).
\end{dfn}
By "maximal" we mean that the applications $f$, $\G z$ and $z$ have the same regularity. Due to  the closed graph theorem,  if  $\G\in \mathscr{MR}_p(0,T;X)$ then
\begin{equation}
\label{max_reg_estim}
\Vert \dot{z}\Vert_{L^p([0,T],X)}+\Vert z\Vert_{L^p([0,T],X)}+\Vert \G z\Vert_{L^p([0,T],X)}\leq C\Vert f\Vert_{L^p([0,T],X)},
\end{equation}
for a constant $C>0$ independent of $f$.

It is known that a necessary condition for the maximal $L^p$-regularity is that $\G$ generates an analytic semigroup. According to De Simon \cite{Simon} this condition is also sufficient if $X$ is a Hilbert space. On the other hand,  it is shown in \cite{Dore} that if  $\G\in \mathscr{MR}_p(0,T;X)$ for one $p\in [1,\infty]$ then $\G\in \mathscr{MR}_q(0,T;X)$ for all $q\in]1,\infty[$. Moreover if $\G\in \mathscr{MR}_p(0,T;X)$ for one $T>0$, then $\G\in \mathscr{MR}_p(0,T';X)$ for all $T'>0$. Hence we simply write $\G\in \mathscr{MR}(0,T;X)$.

\begin{rem}\label{remark:Maximal_Regularity_Operator}
\begin{enumerate}
\item[{\rm(i)}] Let  $\calC([0,T];D(\G))$ be the space of all continuous functions from $[0,T]$ to $D(\G)$, which  is a dense space of $L^p([0,T];X)$. It is know ((see \cite{LutzWeis} (2.a) or \cite{KunstmannWeis} 1.5)) that  $\G$ has maximal $L^p$-regularity on $[0,T]$ if and only if $(\Ss(t))_{t\geq 0}$ is analytic and the operator $\mathscr{R}$ defined by
\begin{equation}
\label{R_op}
(\mathscr{R}f)(t):=\G\int_0^t \Ss(t-s)f(s)ds \qquad \qquad f\in \calC([0,T],D(\G)),
\end{equation}
extends to a bounded operator on $L^p([0,T];X)$. As we will see in our main results, this characterization is very useful if one works in general Banach spaces.
\item [{\rm(ii)}] It is known (see \cite{Dore}) that if $\G\in \mathscr{MR}(0,T;X)$ then for every $\lambda\in \mathbb{C}$, $\G+\lambda\in \mathscr{MR}(0,T;X)$, hence without lost of generality, we will assume through this paper that our generators satisfy $\omega_0(\G)<0$.
\end{enumerate}
\end{rem}
In order to recall another characterization of maximal regularity, we need some definitions.

\begin{dfn}
We say that a Banach space $X$ is a UMD-space if for some (hence all) $p\in (1,\infty)$ , $\mathscr{H}\in \mathcal{L}(L^p(\mathbb{R},X))$ where
$$
(\mathscr{H}f)(t) = \frac{1}{\pi} \lim_{\epsilon\to 0} \int_{\vert s\vert>\epsilon} \frac{f(t-s)}{s}ds,\quad t\in \mathbb{R},\quad  f\in \mathcal{S}(\mathbb{R},X),
$$
where $\mathcal{S}(\mathbb{R},X)$ is the Schwartz space.
\end{dfn}
Classical  UMD-spaces are Hilbert spaces and $L^p$-spaces, where $p\in (1,\infty)$. It is to be noted that every  UMD-space is a reflexive space (see \cite{Amann}).
\begin{dfn}
A set $\tau \subset \mathcal{L}(X,Y)$ is called $\mathcal{R}$-bounded if there is a constant $C>0$ such that for all $n\in \mathbb{N}$, $T_1,...,T_n \in \tau$, $x_1,...,x_n \in X$,
\begin{equation*}
\int_0^1 \Vert \sum_{j=1}^n r_j(s) T_j x_j \Vert_Y ds \leq C \int_0^1 \Vert \sum_{j=1}^n r_j(s) x_j \Vert_X ds
\end{equation*}
where $(r_j)_{j\geq 1}$ is a sequence of independent $\{-1;1\}$-valued random variables on $[0,1]$(e.g. Rademacher variables).
\end{dfn}

\begin{rem}\label{Transfer_Function_R-boundedness}
Here we give examples of $\mathcal{R}$-bounded sets.  We let $A$ be the generator of a bounded  analytic semigroup on a Banach space $X$, $B:U\to X_{-1}$ and $C:D(A)\subset X\to U$ are linear bounded operators, where $U$ is another (boundary) Banach space. We assume that $(A,B,C)$ generates a regular linear system on $X,U,U$ with transfer function
\begin{align*}
H(\la):=C_\Lambda R(\la,A_{-1})B\in\calL(U),\qquad \la\in\rho(A),
\end{align*}
where $C_\Lambda$ is the Yosida extension of $C$ with respect to $A$, see Section \ref{sec:2}. It is shown in \cite[p.513]{HaaKun} that the set $\{H(is):s\neq 0\}$ is $\mathcal{R}$-bounded.
\end{rem}
The following result is due to Weis \cite{LutzWeis2}
\begin{thm}\label{theorem:Weis_Characterization}
Let $\G$ be the generator of a bounded analytic semigroup in a UMD-space $X$. Then $\G$ has maximal $L^p$-regularity for some (all) $p\in (1,\infty)$ if and only if the set $\{sR(is,\G);s\neq 0 \}$ is $\mathcal{R}$-bounded.
\end{thm}

The following remark will be useful in the last section

\begin{rem}\label{remark-for-example}

Let $X,Z,U$ be a Banach spaces such that $Z\subset X$ with dense and continuous embedding, $A_m:Z\to X$ be a closed differential operator and $G:Z\to U$ be a linear surjective operator.  We assume that the following operator
\begin{align*} A:=A_m,\qquad D(A):=\ker(G)
\end{align*}
generates a strongly continuous semigroup $\T:=(\T(t))_{t\ge 0}$ on $X$. let $\D_\la$ the Dirichlet operator associated with $A$ and $G$ (see Section \ref{sec:3}). Moreover, we assume that the following operator
\begin{align*}
B:=(\la-A_{-1})\D_\la\in\calL(U,X_{-1})
\end{align*}
is a $p$--admissible control operator for $A$. In addition, we assume that $A$ has the maximal $L^p$-regularity on $X$.

Let us first show that the operator $(-A)^\theta$, for some $\theta\in (0,\frac{1}{p})$, coincides with its Yosida extension with respect to $A$, that is:
\begin{align*}
D([(-A)^\theta]_\wedge) = D((-A)^\theta) \quad and\quad [(-A)^\theta]_\wedge = (-A)^\theta
\end{align*}
Since it is clear that  $D((-A)^\theta)\subset D([(-A)^\theta]_\wedge)$, it remains only to prove that $D([(-A)^\theta]_\wedge)\subset D((-A)^\theta)$. Then let $x\in D([(-A)^\theta]_\wedge)$ and set $y:=[(-A)^\theta]_\wedge x$ and $y_\lambda=(-A)^\theta\lambda R(\lambda,A)x$ for $\lambda$ sufficiently large. Clearly we have $y_\lambda\to y$ and $(-A)^{-\theta} y_\lambda = \lambda R(\lambda,A)x \to x$ as $\lambda\to +\infty$. Then by closeness of $(-A)^{-\theta}$,  $x=(-A)^{-\theta}y$.  Thus $x\in D((-A)^\theta)$. This ends proof.

Finally, let us show that the  triple $(A,B,(-A)^\theta)$ generates a regular system. In fact, we first prove that $range(\mathbb{D}_\mu)\subset D((-A)^\theta)$ (which is equivalent to the regularity of the system generated by $(A,B,(-A)^\theta)$) for $\mu$ sufficiently large:
We know that if $B$ is $p$-admissible then  $range(\mathbb{D}_\mu)\subset F^A_\frac{1}{p}$. Since $F^A_\frac{1}{p}\subset D((-A)^\theta)$ (see \cite{EngNag}), we have $range(\mathbb{D}_\mu)\subset D((-A)^\theta)$ and the closed graph theorem asserts that $(-A)^\theta \mathbb{D}_\mu \in \mathcal{L}(U,X)$.

By virtue of analyticity of the semigroup generated by $A$, $((-A)^\theta,A)$ are $p$-admissible. To show the well-posedness of the system generated by $(A,B,(-A)^\theta)$ we have only to show that the operator $\mathbb{F}_\infty$ defined by:
\begin{align*}
(\mathbb{F}_\infty u)(t) := (-A)^\theta \Phi_t u,\qquad u\in W^{2,p}_{0,loc}([0,\infty),U)
\end{align*}
is well defined and extends to a bounded operator on $L^p_{loc}([0,\infty),U)$. In fact, by integration by parts and assuming that $0\in \rho(A)$ we have
\begin{align*}
\Phi_t u = \mathbb{D}_0 u(t) - \int_0^t \T(t-s)\mathbb{D}_0 u'(s)ds.
\end{align*}
This show that $\Phi_t u\in D((-A)^\theta)$ since $\int_0^t T(t-s)\mathbb{D}_0 u'(s)ds \in D(A)$ for a.e $t\geq 0$. Now we show the boundedness of $\mathbb{F}_\infty$. We have
\begin{align*}
(\mathbb{F}_\infty u)(t) &= (-A)^\theta \int_0^t (-A)^{1-\theta}\T(t-s)(-A)^{\theta}\mathbb{D}_0u(s)ds\\
&= -A\int_0^t \T(t-s)(-A)^{\theta}\mathbb{D}_0u(s)ds\\
&=-(\mathcal{R}(-A)^{\theta}\mathbb{D}_0u)(t)
\end{align*}
Maximal regularity of $A$ shows the boundedness of $\mathbb{F}_\infty$. This finishes the proof.

\end{rem}

\subsection{Perturbations that are $p$-admissible observation operators}\label{sub:4.1}
In this part, we investigate maximal $L^p$-regularity for the problem \eqref{PM} in the case $K=0$. This is equivalent to study a such property for the evolution equation \eqref{AP-only}. As we have seen in the introductory section, we continue to  assume that $P:Z\subset X\to X$ and $A:=A_m$ with domain $D(A)=\ker(G)$ is the generator of strongly continuous semigroup $\T:=(\T(t))_{t\ge 0}$ on $X$. We define $\P:=P\imath$ with $\imath:D(A)\to X$ is a continuous injection. So that $\P\in\calL(D(A),X)$. We recall from Remark \ref{Hadd_thm} (i) that if $\P$ is a $p$-admissible observation operator for $A,$ then the following operator $A^P:=A+\P=(A+P)\imath$ with domain $D(A^P):=D(A)$ is the generator of a strongly continuous semigroup $\T^p:=(\T^P(t))_{t\ge 0}$ on $X$ such that $\T^P(t)X\subset D(\P_\Lambda)$ for a.e. $t>0,$ and
\begin{align*}
\T^p(t)x=\T(t)x+\int^t_0\T(t-s)\P_\Lambda \T^P(s)xds,
\end{align*}
for all $x\in X$ and $t\ge 0,$ where $\P_\Lambda$ is the Yosida extension of $\P$ with respect to $A$. On the other hand, as shown in \cite{Hadd} for any $f\in L^p([0,T])$ with $T>0,$ the mild solution of the evolution equation \eqref{AP-only} satisfies $z(s)\in D(\P_\Lambda)$ for a.e. $s\ge 0,$
\begin{align}\label{z-MV-formulat}
\begin{split}
z(t)&=\int^t_0 \T^P(t-s)f(s)ds\cr &= \int^{t}_0 \T(t-s)\left(\P_\Lambda z(s)+f(s)\right)ds,
\end{split}
\end{align}for any $t\ge 0$. In addition if we denote by $\P_{\Lambda,A^P}$ the Yosida extension of $\P$ with respect to $A^p,$ then $\P_{\Lambda,A^P}=\P_\Lambda$ on $D(\P_\Lambda)$. So by using \eqref{z-MV-formulat} and Lemma \ref{Hadd_lemma}, there exists a constant $c_T>0$ independent of $f$ such that
\begin{align}\label{need-MV}
\|\P_{\Lambda}z(\cdot)\|_{L^p([0,T],X)} \le c_T \|f\|_{L^p([0,T],X)}.
\end{align}
We now state the main result of this paragraph.
\begin{thm}\label{theorem:Miyadera}
Let $X$ be a Banach space, $p\in ]1,\infty[$ and  $\P$ a $p$-admissible observation operator for $A$. If $A\in \mathscr{MR}(0,T;X)$ then $A^P\in \mathscr{MR}(0,T;X)$.
\end{thm}
\begin{proof}
Assume that $A\in \mathscr{MR}(0,T;X)$, so that $A$ generates an analytic semigroup on $X$.  This shows that there exists  $\omega \in \mathbb{R}$ such that $\mathbb{C}_\omega:=\{\lambda\in \mathbb{C},Re\lambda>\omega\} \subset \rho(A)$ and for every $\lambda\in \mathbb{C}_\omega$ we have:
\begin{align*}
\Vert R(\lambda,A) \Vert \leq \frac{M}{\vert\lambda-\omega\vert}.
\end{align*}
On the other hand, for $\la\in\rho(A),$
\begin{align*}
\la-A^P=(I-PR(\la,A))R(\la,A).
\end{align*}
By the admissibility of $\P$ for $A,$ there exists $\tilde{M}>0$ such that
\begin{align*}
\|\P R(\la,A)\|\le \frac{\tilde{M}}{({\rm Re}\la -\om)^{\frac{1}{q}}},\quad {\rm Re}\la >\om.
\end{align*}
Now for ${\rm Re}\la > \om + 2^q \tilde{M}^q:=\alpha_0$ we have
\begin{align*}
\Vert \P R(\lambda,A) \Vert \leq \frac{1}{2}
\end{align*}
thus $(I-\P R(\lambda,A))$ is invertible and $\Vert(I-\P R(\lambda,A))^{-1}\Vert\leq 2$. Then
\begin{align*}
R(\lambda,A^P)=R(\lambda,A)(I-PR(\lambda,A))^{-1}.
\end{align*}
Finally, for $\lambda\in \mathbb{C}_{\alpha_0}$ we have
\begin{align*}
\Vert R(\lambda,A^p) \Vert \leq \frac{2M}{\vert\lambda-\om\vert}.
\end{align*}
This implies, by \cite[Thm.12.13] {ReRo}; that $(\T^P(t))_{t\geq0}$ is analytic. We now define,  for any $f\in \calC([0,T],D(A)),$
\begin{align*}
(\mathscr{R}f)(t):=A\int_0^t \T(t-s)f(s)ds,\qquad (\mathscr{R^P}f)(t):=A^P\int_0^t \T^P(t-s)f(s)ds,\qquad t\in [0,T].
\end{align*}
Due to \eqref{z-MV-formulat}, we obtain
\begin{align*}
\mathscr{R^P}f=\mathscr{R}(\P_\Lambda z(\cdot)+f)+\P \int^{\cdot}_0 \T(t-s)\left(\P_\Lambda z(s)+f(s)\right)ds.
\end{align*}
Using Remark \ref{remark:Maximal_Regularity_Operator} (i), the estimate \eqref{need-MV} and Lemma \ref{Hadd_lemma}; there exists a constant $\tilde{c}_T>0$ independent of $f$ such that
\begin{align*}
\|\mathscr{R^P}f\|_{L^p([0,T],X)}\le \tilde{c}_T \|f\|_{L^p([0,T],X)}.
\end{align*}
This ends the proof, due to  Remark \ref{remark:Maximal_Regularity_Operator}.
\end{proof}
\begin{rem}\label{Weis-Cor.4}
\begin{enumerate}
  \item In the proof of Theorem \ref{theorem:Miyadera}, we have proved that for $p$-admissible observation operators $\P$ for $A$, the operator $A$ generates an analytic semigroup on a Banach space $X$ if and only if it is so for the operator $A^p$. Hence if $X$ is a Hilbert space, the maximal $L^p$-regularity of $A^P$ is automatically guaranted  by \cite{Simon}.
  \item As explained in Remark \ref{Hadd_thm} (i), $p$-admissible observation operators are also Miyadera-Voigt perturbations operators  for $A$. We mention that the authors of \cite[Cor.4]{KunWei} have obtained a result  on maximal $L^p$-regularity under Miyadera-Voigt perturbations, where it is assumed that the state space $X$ is reflexive (or UMD) and the perturbation $\P$ is a closed and densely defined operator  and satisfies a very special Miyadera-Voigt condition. In our Theorem \ref{theorem:Miyadera}, $X$ is supposed to be a general Banach space and the perturbation $\P$ is not closed and then with even  minimum conditions we have obtained the maximal $L^p$-regularity for $A^P$.
\end{enumerate}
\end{rem}
In the sequel we will also compare our result Theorem \ref{theorem:Miyadera} with a result in \cite[Thm.1]{KunWei} about small perturbations. To that purpose we need the following  lemma.

Let $C\in\calL(D(A),Y)$ for a Banach space $Y$. Define on $D(A)$ the operator
$$
\mathbb{J}x = \frac{1}{2i\pi} \int_\Gamma (-\mu)^{-\beta} CR(\mu,A)x d\mu,
$$
here $\Gamma:= \Gamma(\psi,\epsilon) =\Gamma^1(\psi,\epsilon)\cup \Gamma^2(\psi,\epsilon)\cup\Gamma^3(\psi,\epsilon) $ denotes the upwards oriented path defined by
\begin{align*}
\Gamma^1(\psi,\epsilon) &= \{\lambda\in \mathbb{C}:\vert\lambda\vert\geq\epsilon,arg\lambda=-\psi \}\\
\Gamma^2(\psi,\epsilon) &= \{\lambda\in \mathbb{C}:\vert\lambda\vert=\epsilon,\vert arg\lambda\vert>\psi \}\\
\Gamma^3(\psi,\epsilon) &= \{\lambda\in \mathbb{C}:\vert\lambda\vert\geq\epsilon,arg\lambda=\psi \},
\end{align*}
for $\psi \in (\frac{\pi}{2},\pi)$.
\begin{lem}\label{lemma:observation_fractional_power}
Assume that $A$ is a sectorial operator and $\beta>\frac{1}{p}$ and let  $C\in \mathcal{L}(D(A),Y)$ such that \begin{equation}\label{CR(z,A)_estimate}
\Vert CR(\mu,A) \Vert \leq \frac{M}{({\rm Re} \mu)^\frac{1}{q}},
\end{equation}
for some constant $M>0$ and $\mu$ in some half plane. Then for every $x\in D(A)$, the integral $$\frac{1}{2i\pi} \int_\Gamma (-\mu)^{-\beta} CR(\mu,A)x d\mu,$$ exists as a Bochner integral, the operator $\mathbb{J}$ can be extended to a bounded operator from $X$ to $Y$. Moreover,  for all $x\in D(A)$, $\mathbb{J} x = C(-A)^{-\beta}x$.
\end{lem}
\begin{proof}
Fix $x\in D(A)$ and define $g(\mu)=(-\mu)^{-\beta}CR(\mu,A)x$, $\mu\in \Gamma$. Then we get
\begin{equation}\label{CR(z,A)}
CR(\mu,A)x = \frac{Cx}{\mu} + \frac{CR(\mu,A)Ax}{\mu}
\end{equation}
Cauchy's Theorem applied to the half plane yields that
$$
\int_\Gamma (-\mu)^{-(\beta +1)}d\mu = 0.
$$
It follows by (\ref{CR(z,A)}) that
$$
\frac{1}{2i\pi} \int_\Gamma (-\mu)^{-\beta} CR(\mu,A)x d\mu = \frac{1}{2i\pi} \int_\Gamma (-\mu)^{-(\beta + 1)} CR(\mu,A)Ax d\mu,
$$
which exists as a Bochner integral due to \eqref{CR(z,A)_estimate}. Let $\Gamma_n:= \Gamma_n(\psi,\epsilon) =\Gamma_n^1(\psi,\epsilon)\cup \Gamma^2(\psi,\epsilon)\cup\Gamma_n^3(\psi,\epsilon) $ denotes the upwards oriented path defined by
\begin{align*}
\Gamma_n^1(\psi,\epsilon) &= \{\lambda\in \mathbb{C}:n\geq\vert\lambda\vert\geq\epsilon,arg\lambda=-\psi \}\\
\Gamma_n^3(\psi,\epsilon) &= \{\lambda\in \mathbb{C}:n\geq\vert\lambda\vert\geq\epsilon,arg\lambda=\psi \}
\end{align*}
It follows from Cauchy's Theorem that
\begin{equation}\label{limCR(z,A)}
\frac{1}{2i\pi} \int_\Gamma (-\mu)^{-\beta} CR(\mu,A)x d\mu = \lim_{n\to\infty} \frac{1}{2i\pi} \int_{\Gamma_n} (-\mu)^{-\beta} CR(\mu,A)x d\mu.
\end{equation}
Let $\mathcal{C}_n:=\mathcal{C}_n(\psi)$ and $\overline{\Gamma}^2:=\overline{\Gamma}^2(\psi,\epsilon)$ be the upwards oriented curve defined by
$$
\mathcal{C}_n:=\{\lambda\in \mathbb{C}:\vert \lambda\vert = n,\vert arg \lambda\vert \leq\psi\}
$$
$$
\overline{\Gamma}^2(\psi,\epsilon):=\{\lambda\in \mathbb{C}:\vert \lambda\vert = \epsilon,\vert arg \lambda\vert \leq\psi\}.
$$
Using again Cauchy's Theorem we get
\begin{align*}
&\frac{1}{2i\pi} \int_{\Gamma_n^1\cup\Gamma_n^3} (-\mu)^{-\beta} CR(\mu,A)x dz\\
&= \frac{1}{2i\pi} \int_{\mathcal{C}_n} (-\mu)^{-\beta} CR(\mu,A)x dz +\frac{1}{2i\pi} \int_{\overline{\Gamma}^2} (-\mu)^{-\beta} CR(\mu,A)x d\mu
\end{align*}
Now, we are going to estimate the integrals over ${\Gamma}^2$, $\overline{\Gamma}^2$ and $\mathcal{C}_n$. We start with the integral over $\mathcal{C}_n$. Then we obtain
$$
\frac{1}{2i\pi} \int_{\mathcal{C}_n} (-\mu)^{-\beta} CR(\mu,A)x d\mu = \frac{1}{2\pi} \int_{-\psi}^\psi (n e^{i(\theta-\pi)})^{-\beta} CR(n e^{i\theta},A)n e^{i\theta}x d\theta.
$$
The fact that $C$ satisfies (\ref{CR(z,A)_estimate}), we obtain
\begin{align*}
\left\| \frac{1}{2i\pi} \int_{\mathcal{C}_n} (-\mu)^{-\beta} CR(\mu,A)x d\mu \right\| &\leq  \frac{1}{2\pi} \int_{-\psi}^\psi \Vert (n e^{i(\theta-\pi)})^{-\beta} CR(n e^{i\theta},A)n e^{i\theta}x \Vert d\theta\\
&\leq \frac{M}{2\pi n^{\beta-\frac{1}{p}}} \int_{-\psi}^\psi \frac{d\theta}{cos(\theta)^\frac{1}{q}}\Vert x\Vert .
\end{align*}
The same estimate holds for the integral over $\overline{\Gamma}^2$. Finally for the integral over ${\Gamma}^2$, we have
\begin{align*}
\left\| \frac{1}{2i\pi} \int_{{\Gamma}^2} (-\mu)^{-\beta} CR(\mu,A)x d\mu \right\| &\leq  \frac{1}{2\pi} \int_{{\Gamma}^2} \vert(-\mu)^{-(\beta+1)}\vert\, \Vert CR(\mu,A)\Vert\Vert x\Vert  d\mu\\
&\leq \frac{M_\epsilon}{\epsilon^\frac{1}{q}} \Vert x\Vert.
\end{align*}
since ${\Gamma}^2$ is a compact set and $\mu\to CR(\mu,A)$ is analytic on ${\Gamma}^2$. Now, putting everything together, we find that there is a constant $\kappa$ not depending on $n$ such that
$$
\left\| \frac{1}{2i\pi} \int_{\mathcal{C}_n} (-\mu)^{-\beta} CR(\mu,A)x d\mu \right\| \leq \kappa \Vert x\Vert
$$
From \eqref{limCR(z,A)} we obtain
$$
\Vert \mathbb{J}x\Vert_Y \leq \kappa\Vert x\Vert,for\ all\ x\in D(A)
$$
Therefore, $\mathbb{J}$ extends to a bounded linear operator on $X$.\\
Next we will show that for $x\in D(A)$ we have $\mathbb{J}x = C(-A)^{-\beta}x$. This is equivalent to show that the operator $C$ and the integral $\int_\Gamma (-\mu)^{-\beta}R(\mu,A)xd\mu$ commute. Since $$(-A)^{-\beta} x= \frac{1}{2i\pi}\int_\Gamma (-\mu)^{-\beta}R(\mu,A)xd\mu,$$ $A$ commutes with $\int_\Gamma (-\mu)^{-\beta}R(\mu,A)xd\mu$ for every $x\in D(A)$. Now let us show that $\int_{\Gamma_n} (-z)^{-\beta}R(\mu,A)xd\mu$ converges in $D(A)$. The closedness of $A$ yields
$$
A\int_{\Gamma_n} (-\mu)^{-\beta}R(\mu,A)xd\mu = \int_{\Gamma_n} (-\mu)^{-\beta}R(\mu,A)Axd\mu,
$$
thus the integral $\int_{\Gamma_n} (-\mu)^{-\beta}R(\mu,A)xd\mu$ converges in $D(A)$. As $C$ is continuous on $D(A)$, we obtain
\begin{align*}
C(-A)^{-\beta}x &= \lim_{n\to\infty} \frac{1}{2i\pi}C\int_{\Gamma_n} (-\mu)^{-\beta}R(\mu,A)xd\mu\\
&=\lim_{n\to\infty} \frac{1}{2i\pi}\int_{\Gamma_n} (-\mu)^{-\beta}CR(\mu,A)xd\mu\\
&= \mathbb{J}x.
\end{align*}
for all $x\in D(A)$. This ends the proof.
\end{proof}
\begin{rem}\label{comparaison-Luts-small} If $\P\in\calL(D(A),X)$ is $p$-admissible then it verifies the estimate \eqref{CR(z,A)_estimate}. If in addition $A$ is sectorial, then for $\beta>\frac{1}{p}$ the operator $\P(-A)^{-\beta}$ has a bounded extension to $X,$ due to Lemma \ref{lemma:observation_fractional_power}. On the other hand, for any $x\in D(A)$ one can write
\begin{align*}
Px=\P(-A)^{-\beta}(-A)^{\beta}x.
\end{align*}
This implies that there exists a constant $c>0$ such that
\begin{align*}
\|Px\|\le c \|(-A)^{\beta}x\|.
\end{align*}
As $(-A)^{\beta}$ is a small perturbation for $A,$ then $P$ is so. Now by applying \cite[Thm.1]{KunWei}, the operator $A^P$ is sectorial as well. But if $A$ has the maximal $L^p$-regularity, the result of \cite[Thm.1]{KunWei} confirms that $A^p$ has also the maximal $L^p$-regularity only if the state space $X$ is a UMD space. However  Theorem \ref{theorem:Miyadera}  shows that the  maximal $L^p$-regularity  is preserved for $A^p$ even if we work in a  general Banach space. This confirms that the $p$-admissibility for the perturbation operator is a very powerful tool to prove maximal $L^p$-regularity in Banach spaces.
\end{rem}

\subsection{Desch-Schappacher perturbation}\label{sub:4.2}
In this section we will discuss maximal $L^p$-regularity of the perturbed boundary problem \eqref{PM} (or equivalently \eqref{calA-max}) under conditions {\bf(H1)} and {\bf(H2)} as in Section \ref{sec:3}   and when the boundary perturbation $K$ satisfies the condition
\begin{itemize}
\item[] \textbf{(H3)} $K:X\to U$ is linear bounded (i.e. $K\in\calL(X,U)$).
\end{itemize}
On the other hand, let $B$ as in \eqref{B_op}. We shall also consider the following assumption
\begin{itemize}
\item[] \textbf{(H4)} $B$ is a $p$-admissible control operator for $A$.
\end{itemize}
We first study the maximal $L^p$-regularity for the evolution equation \eqref{calA-only}, where the operator $(\calA,D(\calA))$ is defined by \eqref{cald-def}. As $K$ is bounded then, under the above  conditions, the triple operator  $(A,B,K)$ generates a regular linear system on $X,U,U$ with $I_U:U\to U$ as admissible feedback. By Theorem \ref{theorem:Hadd_Manzo_Rhandi}, the operator $\calA$ generates a strongly continuous semigroup $\T^{cl}:=(\T^{cl}(t))_{\ge 0}$ on $X$ and then the unique mild solution \eqref{calA-only} is given by
\begin{align} \label{DS-form-mild0}
z(t)=\int^{t}_0 \T^{cl}(t-s)f(s)ds,
\end{align}
for any $t\ge 0$ and $f\in L^p([0,T],X)$ with $T>0$. According to Proposition \ref{VCF-calA-f}, this mild solution satisfies also
\begin{align}\label{DS-form-mild}
z(t)=\int^t_0\T_{-1}(t-s)BKz(s)ds+\int^t_0 \T(t-s)f(s)ds,\qquad t\ge 0.
\end{align}
\begin{rem}\label{Weis-DS-gene}
Let us assume that $\T$ is an analytic semigroup on $X$ and $B$ satisfies the condition \begin{align}\label{Conj-Weiss-B}\|R(\la,A_{-1})B\|\le\frac{\kappa}{({\rm Re}\la -\om)^{\frac{1}{p}}},\qquad {\rm Re}\la>\om,
\end{align}for any $\om>\om_0(A)$ and a constant $\kappa>0$. Then without assuming the condition \textbf{(H4)}, one can use the same argument as in \cite[Thm.8]{KunWei} (in the case $\al=0$) and the fact that $\calA$ coincides with the part of the operator $A_{-1}+BK$ on $X$ to prove that $\calA$ generates an analytic semigroup on $X$. In the absence of analyticity of $\T$ one cannot prove this generation result. Observe that our condition \textbf{(H4)} implies the estimate \eqref{Conj-Weiss-B}, see e.g. \cite[chap.3]{TucWei}. With conditions {\bf(H1)} to \textbf{(H4)} we have showed that $\calA$ is a generator on $X$ without assuming any analyticity of $\T$, see Remark \ref{Hadd_thm} (ii).
\end{rem}
\begin{thm}\label{Desch_SchapThm}
Assume that $X$ is a Banach state space and  that conditions {\bf(H1)} to {\bf(H4)} are verified. If $A\in \mathscr{MR}(0,T;X)$ then $\calA \in \mathscr{MR}(0,T;X)$ (and hence the evolution equation \eqref{calA-only}  has the maximal $L^p$-regularity).
\end{thm}

\begin{proof}
Let us first show that the semigroup $\T^{cl}$ generated by $\calA$ is analytic. The condition  $A\in \mathscr{MR}(0,T;X)$ implies that the semigroup $\T$ is analytic. Hence, by \cite[Thm. 12.31]{ReRo}, we can find $\om>\om_0(A)$ and a constant $c>0$ such that
\begin{align}\label{used-here}
\Vert R(\lambda,A)\Vert \leq \frac{c}{\vert \lambda - \omega \vert},
\end{align}
for and $\la\in\C$ such that ${\rm Re}\la>\om$. Now due to \eqref{Conj-Weiss-B}, for ${\rm Re}\lambda > \om+(2\kappa\|K\|)^p=:\tilde{\om},$ we have
\begin{align*}
1\in\rho(R(\la,A_{-1})BK)\quad \text{and}\quad \left\|(I-R(\la,A_{-1})BK)^{-1}\right\|\le 2.
\end{align*}
According to Theorem \ref{theorem:Hadd_Manzo_Rhandi}, we have $\{\la\in\C: {\rm Re}\lambda>\tilde{\om}\}\subset \rho(\calA)$ and
\begin{equation*}
\Vert R(\lambda,\calA)\Vert \leq \frac{2c}{\vert \lambda - \tilde{\om}\vert},
\end{equation*}
for some constant ${\rm Re}\lambda>\tilde{\om}$, due to \eqref{used-here}. This shows that  $(\sgT^{cl} (t))_{t\geq 0}$ is analytic, by \cite[Thm 12.31 ]{ReRo}. Now define the following linear operators
$$
(\mathscr{R}^{cl}f)(t):=\calA\int_0^t \sgT^{cl}(t-s)f(s)ds \qquad\text{and}\quad (\mathscr{R}f)(t):=A\int_0^t \sgT(t-s)f(s)ds,
$$
for almost every  $t\in [0,T]$ and  all $f\in C([0,T],D(\calA))$. Combining \eqref{DS-form-mild0}, \eqref{3-calA-f-z} together with \eqref{cald-def}, for almost every $t\in [0,T]$ and all  $f\in C([0,T],D(\calA))$ we obtain
\begin{align}\label{EE}
\mathscr{R}^{cl}f=A_m \int^t_0\T_{-1}(t-s)BKz(s)ds+(\mathscr{R}f)(t).
\end{align}
On the other hand, taking in to account that the function $z$ is the solution of the evolution equation \eqref{calA-only},   using an integration by parts and the fact that ${\rm range}(\D_\mu)\subset \ker(\mu-A_m)$ for any $\mu\in\rho(A)$  we have
\begin{align*}
A_m \int^t_0\T_{-1}(t-s)BKz(s)ds&=\mu \left(\mathscr{R}( Kz(\cdot)\right)(t)+ A_m \int^t_0\T_{-1}(t-s)(-A_{-1})\D_\mu  Kz(s)ds\cr &= \mu \left(\mathscr{R}( Kz(\cdot)\right)(t)+ A \int^t_0\T (t-s) \D_\mu K (\calA z(s)+f(s))ds,
\end{align*}
for almost every $t\ge 0,$ and all $f\in C([0,T],D(\calA))$. Now the identity \eqref{EE} becomes
\begin{align}\label{EEE}
(\mathscr{R}^{cl}f)(t)=(\mathscr{R}g_\mu)(t)+ \left(\mathscr{R}( \D_\mu K \mathscr{R}^{cl}f)\right)(t)
\end{align}
for almost every $t\ge 0,$ all $\mu\in \rho(A)$ and $f\in C([0,T],D(\calA))$, where  $$g_\mu:= (I+\D_\mu K)f+ \mu Kz(\cdot).$$
By assumption there exists $c_T>0$ such that
\begin{align*}
\left\|\mathscr{R}\varphi \right\|_{L^p([0,T],X)}\le c_T\,\|\varphi\|_{L^p([0,T],X)},\qquad \varphi\in {L^p([0,T],X)}.
\end{align*}
Let $\om>\max\{\om_0(A),\om_0(\calA)\}$ and choose and fix $\mu>\om+(2c_T \kappa \|K\|)^p,$ where the constant $\kappa>0$ is given in \eqref{Conj-Weiss-B}.
 Then we have
 \begin{align}\label{est-6}
 \|\D_\mu K\|\le \frac{1}{2c_T}\quad.
 \end{align}
Now using \eqref{est-6}, \eqref{DS-form-mild0} and H\"{o}lder inequality, we obtain
\begin{align}\label{g-mu-est}
\|g_\mu\|_{L^p([0,T],X)} \le \left(1+\frac{1}{2c_T}+ T|\mu|\|K\| e^{|\om|T})\right)\, \|f\|_{L^p([0,T],X)}:=\tilde{c}_T \, \|f\|_{L^p([0,T],X)}.
\end{align}
Now we define the operator
\begin{align*}
(\mathscr{F}^\mu g)(t)= \mathscr{R} (\D_\mu K g)(t),
\end{align*}
for any $t\ge 0$ and measurable functions $g:[0,T]\to X$.  Using \eqref{EEE}, we obtain
\begin{align}\label{final-ident}
(I-\mathscr{F}^\mu)\mathscr{R}^{cl}f= \mathscr{R} g_\mu\quad.
\end{align}
Remark that the restriction of  $I-\mathscr{F}^\mu$ to $L^p([0,T],X)$ is invertible, since by \eqref{est-6}, we have
\begin{align*}
\|\mathscr{F}^\mu g\|_{L^p([0,T],X)}\le \frac{1}{2} \|g\|_{L^p([0,T],X)}.
\end{align*}
Now as $\mathscr{R} g_\mu\in L^p([0,T],X)$ then by \eqref{final-ident}, we have $\mathscr{R}^{cl}f\in L^p([0,T],X)$ and
\begin{align*}
\mathscr{R}^{cl}f=\left(I-\mathscr{F}^\mu\right)^{-1}\mathscr{R} g_\mu.
\end{align*}
Finally, using \eqref{g-mu-est}, we obtain
\begin{align*}
\|\mathscr{R}^{cl}f\|_{L^p([0,T],X)}\le 2 c_T \tilde{c}_T \|f\|_{L^p([0,T],X)},\qquad f\in \calC([0,T], D(\calA)).
\end{align*}
The required result now follows by density.

\end{proof}

\begin{rem}\label{comparaison-DS-KW} In \cite[Rem.11]{KunWei}, the authors showed that if $A$ has a maximal $L^p$-regularity on a UMD space $X$ and a perturbation $P:X\to X_{-1}$ satisfies $\|(-A_{-1})^{-1}P\|\le \eta$ with $\eta$ small in some sense (see condition (7) in  \cite{KunWei}), then the part of $A_{-1}+P$ on $X$ has also the maximal $L^p$--regularity on $X$. The UMD property is an essential condition in \cite{KunWei} due to a Weis' perturbation theorem \cite{LutzWeis2}.  In our case, $X$ is a general Banach space (not necessarily UMD). However, instead of the above condition on $P$ we have assumed  that the operator $P=BK$ is $p$--admissible control operator for $A$ (which is the case when $B$ is so). This condition together with   \eqref{Conj-Weiss-B} easily imply the condition (7) in \cite{KunWei}.
\end{rem}
We now state the result giving the maximal $L^p$-regularity for the systems \eqref{PM} (or equivalently for the equation \eqref{calA-max}) in the case when $K\in\calL(X,U)$. This is equivalent to the
\begin{thm}\label{max-reg-DS+MV}
Let $X,Z,U$ be Banach spaces such that $Z\subset X$ (with continuous and dense embedding ), $p\in (1,\infty)$ and consider the evolution equation \eqref{PM} with bounded boundary perturbation operator $K\in\calL(X,U)$. Assume that the conditions {\bf(H1)} to {\bf(H4)} are satisfied. Moreover, we assume that the triple $(A,B,\P)$ generates a regular linear system on $X,U,X$ where $\P$ is the restriction of $P$ to $D(A)$. Then the operator $(\calA+P,D(\calA))$ is the generator of a strongly continuous semigroup on $X$. Moreover, if $A\in \mathscr{MR}(0,T;X)$ then $\calA+P \in \mathscr{MR}(0,T;X)$ (and hence the evolution equation \eqref{PM}  has the maximal $L^p$-regularity).
\end{thm}
\begin{proof}
The fact that $(\calA+P,D(\calA))$ is a generator on $X$  is already proved in Theorem \ref{theorem:generalisation_of_Hadd_Manzo_Rhandi}. Now if $A\in \mathscr{MR}(0,T;X)$, then $\calA\in \mathscr{MR}(0,T;X)$, by Theorem \ref{Desch_SchapThm}. On the other hand, Theorem \ref{theorem:generalisation_of_Hadd_Manzo_Rhandi} shows that $P$ is $p$-admissible observation operator for $\calA$. So, thanks to Theorem \ref{theorem:Miyadera} we also have  $\calA+P \in \mathscr{MR}(0,T;X)$.
\end{proof}

\subsection{Staffans-Weiss perturbation}\label{sub:4.3}
In this part, we study maximal $L^p$--regularity for the boundary perturbed equation \eqref{PM} in the general case when the boundary perturbation $K$ is unbounded. We then assume, as in the previous part of this paper, that {\bf(H1)} and {\bf(H2)} are satisfied. In addition we suppose the following condition
\begin{itemize}
\item[] \textbf{(H3)'} $K:Z\to U$ is linear bounded (i.e. $K\in\calL(Z,U)$).
\end{itemize}
On the other hand, let $B$ and $C$ as in \eqref{B_op}. We shall also consider the following assumption
\begin{itemize}
\item[] \textbf{(H4)'} the triple $(A,B,C)$ generates a regular linear system on $X,U,U$ with $I_U:U\to U$ as an admissible feedback operator.
\end{itemize}
\begin{thm}\label{theorem:staffans-weiss-analytic}
Let $X,Z,U$ be Banach spaces such that $Z\subset X$ (with continuous and dense embedding) and let conditions {\bf(H1)}, {\bf(H2)}, ${ \bf (H3)}'$ and ${ \bf (H4)}'$ be satisfied. Then the operator $(\calA,D(\calA))$ defined by \eqref{cald-def} generates a strongly continuous semigroup which is analytic whenever the semigroup generated by $A$ is.
\end{thm}
\begin{proof}
According to Theorem \ref{theorem:Hadd_Manzo_Rhandi} (i) $\calA$ is a generator of a strongly continuous semigroup $\T^{cl}:=(\T^{cl}(t))_{t\ge 0}$ on $X$. Now assume that $A$ generates an analytic semigroup $\T$ on $X$. Then there exist constants $\beta\in \mathbb{R}$ and $M_1>0$ such that $\mathbb{C}_\beta\subset \rho(A)$ and
\begin{equation}\label{analytic-A-estim}
\left\|R(\lambda,A)\right\| \leq \frac{M_1}{|\lambda-\beta|} ,\qquad \lambda\in \mathbb{C}_\beta
\end{equation}
On the other hand, let us prove that the admissibility of $B$ and $C$ for $A,$ imply that
\begin{align}\label{Good-estim-Analy}
M_2:=\sup_{z\in\mathbb{C}_0}\vert z\vert^\frac{1}{q} \Vert CR(z,A) \Vert < +\infty\quad\text{and}\quad M_3:=\sup_{z\in\mathbb{C}_0}\vert z\vert^\frac{1}{p} \Vert R(z,A_{-1})B \Vert < +\infty,
\end{align}
where $\frac{1}{p}+\frac{1}{q}=1$. In fact,  we will give a slight modification of the proof  given in \cite[lem.1.6]{BouDriElm}. Since $A$ generates a bounded analytic semigroup there exist $\omega \in ]\frac{\pi}{2},\pi[$ such that $\sigma(A) \subset \mathbb{C}\backslash \overline{\Sigma}_\omega$. Let $\gamma \in ]\frac{\pi}{2},\omega[$ and $\Gamma$ the path defined by
$$
\Gamma = \{re^{\pm i \gamma},r>0 \}
$$
We can see easily that $\frac{\vert Re z\vert}{\vert z\vert} = \sin\gamma$. By virtue of the resolvent equation and
using the analyticity of the semigroup, we obtain:
$$
\vert z\vert^\frac{1}{q} \Vert CR(z,A) \Vert \leq M_\Gamma \sup_{s\in\mathbb{C}_0} ({\rm Re}(s))^\frac{1}{q} \Vert CR(s,A) \Vert, z\in \Gamma
$$
for some constant $M_\Gamma>0$. Since $C$ is $p$-admissible for $A$, we have
$$
\sup_{s\in\mathbb{C}_0} ({\rm Re}(s))^\frac{1}{q} \Vert CR(s,A) \Vert < +\infty
$$
and thus:
$$
\sup_{z\in\Gamma}\vert z\vert^\frac{1}{q} \Vert CR(z,A) \Vert < +\infty
$$
The application
\begin{align*}
\varphi : \Sigma_\omega &\rightarrow \mathcal{L}(X,Y)\\
z & \mapsto z^\frac{1}{q} CR(z,A)
\end{align*}
is analytic, then for all $z\in \mathbb{C}_0$ we have
$$
z^\frac{1}{q} CR(z,A) = \frac{1}{2\pi i} \int_\Gamma \frac{u^\frac{1}{q}CR(u,A)}{u-z}du
$$
Since $\varphi$ is bounded on $\Gamma$, then it is bounded on $\mathbb{C}_0$ and this is what we want. The other estimation is obtained by the same arguments. Let $\om_1:=\max\{\om_0(A),\om_0(\calA)\}$. From Theorem \ref{theorem:Hadd_Manzo_Rhandi} (ii) we know that for any $\la\in\C_{\om_1},$ we have \begin{align}\label{Resolvent-calA}
R(\la,\calA)&=(I-\D_\la K)^{-1}R(\la,A)\cr &= R(\la,A)+R(\la,A_{-1})B (I_U-C_\Lambda \D_\la)^{-1}CR(\la,A).
\end{align}
On the other hand, $(I_U-C_\Lambda \D_\la)^{-1}=I+\mathscr{H}^{cl}(\la),$ where $\mathscr{H}^{cl}$ is the transfer function of the (closed-loop) regular linear system generated by $(\calA,B,C_\Lambda)$. Hence there exists $\al>\om_0(\calA)$ such that
\begin{align}\label{transfer-estim}
\nu:=\sup_{\C_\al}\left\|(I_U-C_\Lambda \D_\la)^{-1}\right\|<\infty.
\end{align}
Now let $\om_2:=\max\{0,\alpha,\beta,\om_1\}$. Then by using \eqref{analytic-A-estim}, \eqref{Good-estim-Analy}, \eqref{Resolvent-calA} and \eqref{transfer-estim}, we obtain
\begin{align*}
\|R(\la,A)\|\le \frac{\tilde{M}}{|\la-\om_2|}
\end{align*}
for all $\la\in\C_{\om_2}$, where $\tilde{M}=M_{1}+\nu M_{2}M_{3}$.
 \end{proof}

%

The following result shows the maximal regularity of the perturbed boundary value problem \eqref{PM} in the case of $P\equiv 0$.

\begin{thm}\label{theorem:Staffans_Weiss_Non_Reflexive_Case}
Let $X,Z,U$ be Banach spaces such that $Z\subset X$ (with continuous and dense embedding) and let conditions {\bf(H1)}, {\bf(H2)}, ${ \bf (H3)}'$ and ${ \bf (H4)}'$ be satisfied and let $p\in (1,\infty)$. Let the operator $(\calA,D(\calA))$ defined by \eqref{cald-def}. Assume additionally that there exists $\lambda_0 \in \mathbb{R}$ such that
\begin{align}\label{kappa-0}
\kappa_0:=\sup_{Re\lambda>\lambda_0}\Vert \lambda \mathbb{D}_\lambda \Vert < +\infty
\end{align} If $A\in \mathscr{MR}(0,T;X)$ then $\calA\in \mathscr{MR}(0,T;X)$.
\end{thm}

\begin{proof} As $A\in \mathscr{MR}(0,T;X)$, there exists $\mathscr{R}\in \calL(L^p([0,T],X))$ such that
\begin{align*}
(\mathscr{R}f)(t)=A\int^t_0 \T(t-s)f(s)ds,
\end{align*}
for all $f\in L^p([0,T],X)$ and a.e $t\ge 0$. By Theorem \ref{theorem:staffans-weiss-analytic}, $\calA$ generates an analytic semigroup $\T^{cl}$ on $X$. We then can define the following operator
$$
(\mathscr{R}^{cl}f)(t):=\calA\int_0^t \sgT^{cl}(t-s)f(s)ds \qquad \qquad f\in C([0,T];D(\calA)).
$$
Our objective is to show that  the operator $\mathscr{R}^{cl}$ admits a bounded extension on $L^p([0,T];X)$. On the other hand, we define the Yosida approximation operators of $\calA$ by
\begin{align*}
\calA_n := n\calA R(n,\calA_n)= n^2 R(n,\calA_n) -nI,
\end{align*}
for any $n\in\N$ such that $n>\om_0(\calA)$. From  \eqref{Resolvent-calA} and for any sufficiently  integer $n$,
one can write
\begin{align}\label{calAn}
\calA_n = nAR(n,A) +n^2 \mathbb{D}_n(I-C_\wedge \mathbb{D}_n)^{-1}CR(n,A)\quad .
\end{align}
We also set
$$
(\mathscr{R}^{cl}_nf)(t):=\calA_n\int_0^t \sgT^{cl}(t-s)f(s)ds,
$$
for $f\in C([0,T];D(\calA))$ and $n\in\N$ such that $n>\om_0(\calA)$. We have (see \cite{EngNag}),
\begin{align*}
(\mathscr{R}^{cl}f)(t)=\lim_{n\to\infty} (\mathscr{R}^{cl}_nf)(t),
\end{align*}
 for every $t\in [0,T]$.  Using Proposition \ref{VCF-calA-f}, we have
\begin{align*}
(\mathscr{R}^{cl}_nf)(t)=\calA_n \left(\int^t_0 \T(t-s)f(s)ds+\int^t_0 \T_{-1}(t-s)C_\Lambda z(s)de\right),
\end{align*}
where
\begin{align*}
z(t)=\int^t_0 \T^{cl}(t-s)f(s)ds,\qquad t\ge 0.
\end{align*}
Using  \eqref{calAn}, for large $n$,
\begin{align*}
(\mathscr{R}^{cl}_n f)(t) &=  nAR(n,A)\int_0^t \sgT(t-s)f(s)ds\\
&+n^2 \mathbb{D}_n(I-C_\Lambda \mathbb{D}_n)^{-1}CR(n,A)\int_0^t \sgT(t-s)f(s)ds  \\
&+nAR(n,A)\int_0^t \sgT_{-1}(t-s)BC_\Lambda z(s)ds\\
&+n^2 \mathbb{D}_n(I-C_\Lambda \mathbb{D}_n)^{-1}CR(n,A)\int_0^t \sgT_{-1}(t-s)BC_\Lambda z(s)ds\\
&:=I_n^1(t) + I_n^2(t) + I_n^3(t) + I_n^4(t).
\end{align*}
We have
\begin{align*}
\int_0^T \Vert I_n^1(t) \Vert^p dt &=\left\|\mathscr{R}( nR(n,A)f ) \right\|_{L^p([0,T],X)}^p\\
&\leq c_T \left\| f \right\|_{(L^p[0,T],X)}^p,
\end{align*}
for a constant $c_T>0$ independent of $f$. On the other hand,
\begin{align*}
\int_0^T \Vert I_n^2(t) \Vert^p dt &=\int_0^T \Vert n \mathbb{D}_n(I-C_\wedge \mathbb{D}_n)^{-1}CnR(n,A)\int_0^t \sgT(t-s)f(s)ds \Vert^p dt\\ & \le (\kappa \nu)^p \int_0^T \Vert  C\int_0^t \sgT(t-s)nR(n,A)f(s)ds \Vert^p dt\\
&\leq (\kappa \nu \ga_T)^p \Vert f \Vert_{L^p}^p:=c_{T,1}  \Vert f \Vert_{L^p}^p,
\end{align*}
due to \eqref{transfer-estim}, \eqref{kappa-0} and  Lemma \ref{Hadd_lemma}, where $c_{T,1}$ is a constant independent of $f$. We estimate $I^3_n(t)$ by
\begin{align*}
\int_0^T \Vert I_n^3(t) \Vert^p dt &=\int_0^T \Vert A\int_0^t \sgT(t-s)n\D_n C_\Lambda z(s)ds \Vert^p dt\\
&=\left\|\mathscr{R}(n\D_n C_\Lambda z(s)) \right\|_{L^p([0,T],X)}^p\\ & \le c_T \kappa_0^p \left\| C_\Lambda z(s ) \right\|_{L^p([0,T],U)}^p \\
&\leq c_{T,2}   \Vert f \Vert_{L^p}^p,
\end{align*}
due to \eqref{kappa-0} and Proposition \ref{VCF-calA-f}, where $c_{T,2}>0$ is a constant independent of $f$. Similarly,
\begin{align*}
\int_0^T \Vert I_n^4(t) \Vert^p dt &=\int_0^T \Vert n \mathbb{D}_n(I-C_\Lambda \mathbb{D}_n)^{-1}C\int_0^t \sgT(t-s)nR(n,A)BC_\Lambda z(s)ds \Vert^p dt\\ &\le (\nu \kappa_0)^p\int_0^T \Vert  C\int_0^t \sgT(t-s)n \mathbb{D}_nC_\Lambda z(s)ds \Vert^p dt\\
&\leq c_{T,3} \Vert f \Vert_{L^p}^p,
\end{align*}
by \eqref{transfer-estim}, \eqref{kappa-0}, Lemma \ref{Hadd_lemma} and Proposition \ref{VCF-calA-f}, where $c_{T,3}>0$ is a constant independent of $f$.   Finally one can conclude that
\begin{align*}
\int_0^T \Vert (\mathscr{R}_n^{cl} f)(t) \Vert^p dt \leq \tilde{C_p} \Vert f \Vert_{L^p}^p,
\end{align*}
 for some constant $\tilde{C_p}>0$ depending on $p$ and independent of $f$.\\
 Since $\Vert (\mathscr{R}_n^{cl} f)(t) \Vert^p \rightarrow \Vert (\mathscr{R}^{cl} f)(t) \Vert^p$ for all $t\in [0,T]$, we conclude by Fatou's lemma that
\begin{align*}
\int_0^T \Vert (\mathscr{R}^{cl} f)(t) \Vert^p dt &\leq  \liminf \int_0^T \Vert (\mathscr{R}_n^{cl} f)(t) \Vert^p dt\\
&\leq \tilde{C_p} \Vert f \Vert_{L^p}^p.
\end{align*}
Thus $\mathscr{R}^{cl}$ can be extended to a bounded operator on $L^p([0,T];X)$.
\end{proof}
\begin{rem} Here we will show that the result of Theorem \ref{theorem:Staffans_Weiss_Non_Reflexive_Case} holds only in non reflexive Banach spaces. To that purpose, we define, for $\alpha \in (0,1)$, the Favard space of order $\alpha$ associated to $A$ by
$$
F_\alpha^A := \left\lbrace x\in X: \sup_{t>0} \left\lVert\frac{\sgT(t)x-x}{t^\alpha}\right\rVert <\infty \right\rbrace
$$
with norm
$$
\Vert x \Vert_{F_\alpha^A}:= \sup_{t>0} \left\lVert\frac{\sgT(t)x-x}{t^\alpha}\right\rVert.
$$
Now the assumption $\sup_{Re\lambda>\lambda_0}\Vert \lambda R(\lambda,A_{-1})B \Vert < +\infty$ in the previous theorem implies that $range(B) \subset F_1^{A_{-1}}$ (see \cite[Remark 10]{MarBouFadHam}) and it is important to remark that the control operator $B$ is strictly unbounded, i.e. $range(B)\cap X = \{0\}$, since it comes from the boundary (see \cite{Sala}). These facts force us to work in non-reflexive Banach spaces, because if $X$ is reflexive, it is well known that $F_1^{A_{-1}}=X$, thus $range(B)\subset X$ and this cannot be true.
\end{rem}
Now we state the previous theorem in the case of $X$ being a  $GT$-space (e.g. if $X=L^1$ or $X=C(K)$, see for instance \cite{KalWei}), which is a non-reflexive space.
\begin{cor}
Let   $X,Z,U$ be Banach spaces such that $Z\subset X$ (with continuous and dense embedding) and let conditions {\bf(H1)}, {\bf(H2)}, ${ \bf (H3)}'$ and ${ \bf (H4)}'$ be satisfied such that either $X$ or $X^*$ is a $GT$-space.   Let the operator $(\calA,D(\calA))$ be defined by \eqref{cald-def}. Assume additionally that there exist $\lambda_0 \in \mathbb{R}$ such that
\begin{align*}
\kappa_0:=\sup_{Re\lambda>\lambda_0}\Vert \lambda \mathbb{D}_\lambda \Vert < +\infty
\end{align*}
If $A$ is an $H^\infty$-sectorial operator on $X$ with $\omega_H(A)<\frac{\pi}{2}$ then $\calA \in \mathscr{MR}(0,T;X)$.
\end{cor}
\begin{proof}
According the Theorem 7.5 in \cite{KalWei}, if $A$ is an $H^\infty$-sectorial operator on $X$ with $\omega_H(A)<\frac{\pi}{2}$ then $A$ has maximal $L^p$-regularity for all $1<p<\infty$, which implies by Theorem \ref{theorem:Staffans_Weiss_Non_Reflexive_Case} that $A^{cl} \in \mathscr{MR}(0,T;X)$.
\end{proof}

The next theorem present a perturbation result on UMD-spaces.
\begin{thm}\label{theorem:staffans-weiss-R-boundedness}
Let conditions {\bf(H1)}, {\bf(H2)}, ${ \bf (H3)}'$ and ${ \bf (H4)}'$ be satisfied with $X,U$ be UMD-spaces and $A$ generates a bounded analytic semigroup. Assume  that there exists $\om>\max\{\om_0(A);\om_0(\calA)\}$ such that the sets $\{s^\frac{1}{p} R(\om+is,A_{-1})B;s\neq 0\}$ and $\{s^\frac{1}{q} C R(\om+is,A);s\neq 0\}$ are $\mathcal{R}$-bounded. If $A\in \mathscr{MR}(0,T;X)$ then $\calA \in \mathscr{MR}(0,T;X)$.
\end{thm}
\begin{proof}
Assume that $A\in \mathscr{MR}(0,T;X)$ and let $\om>\max\{\om_0(A);\om_0(\calA)\}$ such that the sets $\{s^\frac{1}{p} R(\om+is,A_{-1})B;s\neq 0\}$ and $\{s^\frac{1}{q} C R(\om+is,A);s\neq 0\}$ are $\mathcal{R}$-bounded. Denote $\calA^\om:=-\om+\calA$ and $A^\om:=-\om+A$ with domains $D(\calA^\om)=D(\calA)$ and $D(A^\om)=D(A),$ respectively. These operators are generators of analytic semigroups on $X$. We first observe that $A^\om\in \mathscr{MR}(0,T;X)$. To prove our theorem it suffice to show that $\calA^\om\in \mathscr{MR}(0,T;X)$. Clearly, $\om_0(A^\om)=\om_0(A)-\om<0$ and $\om_0(\calA^\om)=\om_0(\calA)-\om<0$, so   that $i\R\backslash\{0\}\subset \rho(A^\om)\cap\rho(\calA^\om)$. It is not difficult to show that $(A^\om,B,C)$ is also a regular linear system on $X,U,U$ with the identity operator $I_U:U\to U$ as an admissible feedback. Now according to Theorem \ref{staff_wei}, the following operator
\begin{align*}
A^{cl,\om}:=A^\om_{-1}+BC_\Lambda, \quad D(A^{cl,\om})=\left\{x\in D(C_\Lambda): (A^\om_{-1}+BC_\Lambda)x\in X\right\}.
\end{align*}
As $A^\om_{-1}=A_{-1}-\om$, then $D(A^{cl,\om})=D(\calA),$ and $A^{cl,\om}=\calA^\om$ due to Theorem \ref{theorem:Hadd_Manzo_Rhandi} (i).
As in \eqref{Resolvent-calA} we have
\begin{align}\label{formu-is}
\begin{split}
sR(is,\calA^\om) &= sR(is,A^\om) + s^\frac{1}{p}R(is,A_{-1}^\om)B(I-H^\om(is))^{-1} s^\frac{1}{q} CR(is,A^\om)\cr & =sR(is,A^\om) + s^\frac{1}{p}R(\om+is,A_{-1})B(I-H^\om(is))^{-1} s^\frac{1}{q} CR(\om+is,A),
\end{split}
\end{align}
where $H^\om(\la)=C_\Lambda R(\la,A_{-1}^\om)B,\;\la\in\rho(A)$, is the transfer function of the regular linear system generated by $(A^\om,B,C)$. Using the assumptions,  the equation \eqref{formu-is} and Theorem \ref{theorem:Weis_Characterization}, it suffice to show that the set $\{(I-H^\om(is))^{-1}:s\neq 0\}$ is $\mathcal{R}$-bounded. In fact, by Theorem \ref{theorem:Hadd_Manzo_Rhandi} (i) and  the condition ${ \bf (H4)}'$ the triple operator $(\calA^\om,B,C_\Lambda)$ generates a regular linear system with transfer function
\begin{align*}
H^{cl,\om}(is) &= (I-H^\om(is))^{-1}H^\om(is),\qquad s\neq 0,
\end{align*}
which implies that
\begin{align*}
(I-H^\om(is))^{-1}=I_U+H^{cl,\om}(is),\qquad s\neq 0.
\end{align*}
According to Remark \ref{Transfer_Function_R-boundedness}, the set $\{H^{cl,\om}(is):s\neq 0\}$ is $\mathcal{R}$-bounded. Hence $\{(I-H^\om(is))^{-1}:s\neq 0\}$ is $\mathcal{R}$-bounded. This ends the proof.
\end{proof}


\begin{prop}\label{corollary:staffans_weiss_R_boundedness}
Let conditions {\bf(H1)}, {\bf(H2)}, ${ \bf (H3)}'$ and ${ \bf (H4)}'$ be satisfied with $X,U$ be UMD-spaces and $A$ generates a bounded analytic semigroup. Assume that there exist  constants $\om>\max\{\om_0(A),\om_0(\calA)\}$ and $\alpha \in (\frac{1}{p},1)$  such that the set $\{s^\alpha R(\om+is,A_{-1})B;s\neq 0\}$ is $\mathcal{R}$-bounded. If $A\in \mathscr{MR}(0,T;X)$ then $\calA \in \mathscr{MR}(0,T;X)$.
\end{prop}
\begin{proof}
Let the operators $A^\om$ and $\calA^\om$ as in the proof of Theorem \ref{theorem:staffans-weiss-R-boundedness}. Let $s\in \R\backslash\{0\}$ and $\al\in (\frac{1}{q},1)$. According to Theorem \ref{theorem:Weis_Characterization} it suffices to show that the set $\{sR(is,\calA^\om):s\neq 0\}$ is $\mathcal{R}$-bounded. In fact, as in \eqref{formu-is}, we obtain
\begin{align*}
sR(is,\calA^\om)=s R(is,A^\om)+s^\alpha R(is,A_{-1}^\om)B(I-H^\om(is))^{-1} s^{1-\alpha} CR(is,A^\om).
\end{align*}
 By the proof of Theorem \ref{theorem:staffans-weiss-R-boundedness}, we know that the set $\{(I-H^\om(is))^{-1}:s\neq 0\}$ is $\mathcal{R}$-bounded. Now as by assumption  the set $\{s^\alpha R(is,A_{-1}^\om)B;s>0\}$ is $\mathcal{R}$-bounded, it suffices to show that the set $\{s^{1-\alpha} CR(is,A^\om);s>0\}$ is $\mathcal{R}$-bounded. We have
\begin{align*}
s^{1-\alpha} CR(is,A^\om) &= s^{1-\alpha} C(-A^\om)^{-\alpha} (-A^\om)^{\alpha} R(is,A^\om)\\
&= C(-A^\om)^{-\alpha} \left[(-A^\om)^{\alpha} (is-A^\om)^{-\alpha} \right] \left[ s^{1-\alpha}(is-A^\om)^{\alpha-1}\right] .\\
\end{align*}
By \cite{KunWei} Lemma 10, the sets $\{ (-A^\om)^{\alpha} (is-A^\om)^{-\alpha};s>0\}$ and $\{ s^{1-\alpha}(is-A^\om)^{\alpha-1};s>0\}$ are $\mathcal{R}$-bounded. On the other hand, by Lemma \ref{lemma:observation_fractional_power} the operator $C(-A^\om)^{-\alpha}$ has a bounded extension to $X$.  Hence the set $\{s^{1-\alpha} CR(is,A^\om);s>0\}$ is $\mathcal{R}$-bounded. This ends the proof.
\end{proof}

We end this section by the following result given the maximal $L^p$-regularity for the evolution equation \eqref{PM} (or equivalently \eqref{calA-max}).
\begin{cor}\label{SF-big-equa-result}
Let conditions {\bf(H1)}, {\bf(H2)}, ${ \bf (H3)}'$ and ${ \bf (H4)}'$ be satisfied on Banach spaces $X,U$ and $A\in \mathscr{MR}(0,T;X)$. Moreover, we assume that $(A,B,\P)$ generates a regular linear system on $X,U,X$ with $\P\in\calL(D(A),X)$ is the restriction of $P:Z\to X$ on $D(A)$. Then the operator $\calA+P\in \mathscr{MR}(0,T;X)$ (or equivalently the evolution equation \eqref{PM} has maximal $L^p$-regularity) if one of the following conditions hold:
\begin{itemize}
  \item [{\rm (i)}]$X$ is a non reflexive Banach space and there exists $\la_0\in\R$ such that
  \begin{align*}
  \sup_{{\rm Re}\la >\la_0}\|\la\D_\la\|<\infty.
  \end{align*}
  \item [{\rm (ii)}] $X$ and $U$ are UMD spaces, the  sets $\{s^\frac{1}{p} R(is,A_{-1})B;s\neq 0\}$ and $\{s^\frac{1}{q} C R(is,A);s\neq 0\}$ are $\mathcal{R}$-bounded.
  \item [{\rm (ii)}] $X$ and $U$ are UMD-spaces, and $A$ generates a bounded analytic semigroup and  there exist  constants $\om>\max\{\om_0(A),\om_0(\calA)\}$ and $\alpha \in (\frac{1}{p},1)$  such that the set $\{s^\alpha R(\om+is,A_{-1})B;s\neq 0\}$ is $\mathcal{R}$-bounded.
\end{itemize}
\end{cor}
\begin{proof}
First of all the operator $P$ is $p$-admissible observation operator for $\calA,$ due to Theorem \ref{theorem:generalisation_of_Hadd_Manzo_Rhandi}. Now the assertion (i) follows by combining Theorem \ref{theorem:Miyadera} and Theorem \ref{theorem:Staffans_Weiss_Non_Reflexive_Case}. The assertion (ii) follows by combining Theorem \ref{theorem:Miyadera} and Theorem \ref{theorem:staffans-weiss-R-boundedness}. Finally,  the assertion (iii) follows by combining Theorem \ref{theorem:Miyadera} and Proposition \ref{corollary:staffans_weiss_R_boundedness}.
\end{proof}

\section{Applications}\label{sec:5}
The object of this section is to apply our obtained results to solve the problem of maximal $L^p$--regularity for integro-differential equations and boundary integro-differential equations. We then extend some results in \cite{Barta1}.
\subsection{Maximal regularity for free-boundary integro-differential equations}
Let $X_0,U_0,Z_0$ be Banach spaces such that $Z_0\subset X_0$ with continuous and dense embedding and let $q\in (1;\infty)$. We consider the following problem:
\begin{align}
\label{IACP}
\begin{cases}
\dot{\varrho}(t) = A_m\varrho(t)+\displaystyle\int_0^t a(t-s)F \varrho(s)ds +f(t), & t\ge 0,  \\ G\varrho(t)=0,& t\ge 0,\\
\varrho(0) = 0,
\end{cases}
\end{align}
where $A_m:Z_0\to X_0$ is a closed linear differential operator, $F:Z_0\to X_0$ a linear operator, $G:Z_0:\to U_0$ is a (trace) linear boundary operator and a certain measurable function $f:[0,\infty)\to X_0$.

The main purpose of this subsection is to apply our abstract results on maximal regularity developed in Section \ref{sec:4} for the integro-differential equation \eqref{IACP}. We first assume that
\begin{enumerate}
\item [{\bf(A1)}] $\A_0:=A_m$ with domain $D(\A_0):=\ker(G)$ is a generator of a $C_0$--semigroup $\T_0:=(\T_0(t))_{t\ge 0}$ on $X_0$.
\end{enumerate}
On the other hand, we denote
\begin{align*}
F_0:=F\imath,\quad \text{with}\quad \imath: D(\A_0)\to X_0\quad \text{is the continuous injection}.
\end{align*}


We now introduce the Banach product space
\begin{align*}
X:=X_0\times L^q(\R^+,X_0)\quad\text{with norm}\quad \left\|(\begin{smallmatrix} x\\ f\end{smallmatrix})\right\|:=\|x\|+\|f\|_q.
\end{align*}
On the other hand, we consider the matrix operator
\begin{align}\label{matrix-opetra1}
\mathfrak{A} :=\left(\begin{array}{cc}
\A_0 & \delta_0\\
\Upsilon & \frac{d}{ds}
\end{array}\right),\qquad
D(\mathfrak{A})=D(\A_0)\times D(\frac{d}{ds}),
\end{align}
where $\Upsilon x = a(\cdot)F  x$ for $x\in D(\A_0)$. Moreover, we define the function $\zeta: [0,+\infty)\to X$ by
\begin{align}\label{zeta-function}
\zeta(t)=\left(\begin{smallmatrix} f(t)\\ 0\end{smallmatrix}\right),\qquad t\ge 0.
\end{align}
Now to solve the integro-differential equation \eqref{IACP}, we just have to solve the equation
\begin{align}
\label{ACP1}
\begin{cases}
\dot{z}(t) = \mathfrak{A} z(t) + \zeta(t), & t\ge 0, \\
z(0) = (\begin{smallmatrix} 0\\ 0\end{smallmatrix}),
\
\end{cases}
\end{align}
and then the solution $\varrho$ of \eqref{IACP} is none other than the first component of the solution $z$(see \cite{EngNag}).

Let define the left shift semigroup on $L^q(\R^+,X_0)$ by
\begin{align*}
(S(t)f)(s)=f(t+s),\qquad t,s\ge 0.
\end{align*}

We recall that the left shift semigroup $\mathbb{S}:=(S(t))_{t\ge 0}$ is not analytic on $L^q(\R^+,X_0)$, so that the evolution equation associated with this semigroup hasn't the  maximal $L^p$--regularity on $L^q(\R^+,X_0)$. Hence one cannot  expect the maximal regularity for the problem \eqref{ACP1} on the space $X$ as well even if we assume that $\A_0$ has the maximal $L^p$--regularity on $X_0$. To overcome this obstacle we shall work in a small space of $X$ with respect to the second component. That is one looks for subspaces of $L^q(\R^+,X_0)$ in which the left shift semigroup  $\mathbb{S}:=(S(t))_{t\ge 0}$ is analytic. This observation has already been used in \cite{Barta1}, \cite{Barta2} to prove analyticity and maximal regularity of a particular class of Volterra equations. Indeed,  let  $h:\mathbb{R}^+\to \mathbb{R}^+$ be an {\em admissible function}, that is $h$ is an  increasing and convex function such that $h(0)=0$. On the other hand, define the sector
$$
\Sigma_h:=\{\si+i\tau\in \mathbb{C};x>0\quad \vert y\vert <h(\si) \}.
$$

For $q\in(1;\infty)$, we define the {\em Bergman space} of holomorphic $L^q$-integrable functions by:
\begin{equation*}
	B_{h,X_0}^q:=B_\theta^q(\Sigma_h;X_0):=\left\{f:\Sigma_h\rightarrow X_0\ \text{holomorphic}\ ;\ \int_{\Sigma_h}\Vert f(\t+i\si)\Vert_{X_0}^q d\tau d\si<\infty\right\},
\end{equation*}
with
\begin{equation*}
	\|f\|_{B_{h,X_0}^q}:=\left(\int_{\Sigma_h}\Vert f(\t+i\si)\Vert_{X_0}^q d\tau d\si\right)^\frac{1}{q}.
\end{equation*}
The following result is a slight modification of the one proved in \cite[lem.4.3]{Barta3}.
\begin{lem}\label{lemma:bergman_estimate}
Let $s>1$, $p>1$ and $h$ be an admissible function such that:
$$\int_0^1 h(\si)^{1-s}d\si<+\infty.$$
Then for every $q = \frac{ps}{s-1}$ and every $R>0$ there exists $C>0$ such that:
$$
\int_0^R \Vert f(t)\Vert_{X_0}^p dt \leq C \|f\|_{B_{h,X_0}^q}^p,
$$
holds for every $f\in B_{h,X_0}^q$.
\end{lem}
\begin{proof}
The proof of this lemma is similar to the proof of Lemma 4.3 in \cite{Barta3}. We estimate $\Vert f(r)\Vert^p$ using Cauchy's formula and we follow the same steps.
\end{proof}
Hereafter, we will take $p>1$, $h(\si) = \tan(\theta)\si$, $s\in(1,2)$ and $q=\frac{ps}{s-1}$. It is not hard to see that $h$ is an admissible function and $\Sigma_h$ is none other than $\Sigma_\theta$.

In the rest of this paragraph, instead of the space $X,$ we will work on the following subspace
\begin{align*}
X^q:=X_0\times B_{h,X_0}^q,\qquad\left\|(\begin{smallmatrix} x\\f\end{smallmatrix}\right)\|_{X^q}:=\|x\|_{X_0}+ \|f\|_{B_{h,X_0}^q}.
\end{align*}
On the space $B_{h,X_0}^q$, we define the complex derivative $\frac{d}{dz}$ with its natural domain :
$$
D\left(\frac{d}{dz}\right):=\left\{ f\in B_{h,X_0}^q; f' \in B_{h,X_0}^q\right\}
$$
It is shown in \cite{Barta1} that the operator $(\frac{d}{dz},D(\frac{d}{dz}))$ generates an analytic semigroup of translation on the Bergman space $B_{h,X_0}^q$ and has the maximal $L^p$-regularity on $L^q(\R^+,X_0)$ whenever  $X_0$ is an UMD space.

Let us now discuss the maximal regularity for the matrix operator $\mathfrak{A}$ with domain $D(\mathfrak{A})=D(\A_0)\times D\left(\frac{d}{dz}\right)$ on the space $X^q$. We first split the operator $\mathfrak{A}$ as $\mathfrak{A}:=A+P,$ where
$$
A :=\left(\begin{array}{cc}
\A_0 & 0\\
0 & \frac{d}{ds}
\end{array}\right);\quad
D(A)=D(\A_0)\times D\left(\frac{d}{dz}\right),
$$
and
$$
P :=\begin{pmatrix}
0 & \delta_0\\
\Upsilon_0 & 0
\end{pmatrix};\quad
D(P)=D(\A_0)\times D\left(\frac{d}{dz}\right).
$$
where $\Upsilon_0 := \Upsilon_{|D(\A_0)}$. Clearly, the operator $A$ generates on $X^q$  the following strongly continuous semigroup
$$
\T(t)=\begin{pmatrix}
\T_0(t) & 0\\
0    & S(t)
\end{pmatrix}.
$$
The next result introduce conditions for which the operator $P$ becomes an admissible observation operator for $A$.
\begin{lem}\label{Lemma-adm-calP}
Assume that $a(\cdot)\in B_{h,\C}^q$ and $F_0\in\calL(D(\A_0),X_0)$ is a $p$-admissible observation operator for $\A_0$. Then $\Upsilon_0\in\calL(D(\A_0),B_{h,X_0}^q)$ is a $p$-admissible operator for $\A_0$ as well. In particular the operator $P\in\calL(D(A),X^q)$ is a $p$-admissible observation operator for $A$.
\end{lem}
\begin{proof}
Since $F_0$ is $p$-admissible observation operator for $\A_0$, for $\al>0,$ there exists $\ga>0$ such that
\begin{align*}
\int^\al_0 \|F_0\T_0(t)x\|_{X_0}^pdt\le \ga^p \|x\|_{X_0}^p ,
\end{align*}
for any $x\in D(\A_0)$. On the other hand, we have
\begin{align*}
\int^\al_0 \|\Upsilon_0 \T_0(t)x\|^p_{B_{h,X_0}^q}dt&=\int^\al_0 \left(\int_{\Sigma_h}\Vert a(\t+i\si)F_0\T_0(t)x\Vert_{X_0}^q d\tau d\si\right)^\frac{p}{q}dt\cr &= \int^\al_0 \left(\int_{\Sigma_h}|a(\t+i\si)| d\tau d\si\right)^\frac{p}{q} \|F_0\T(t)x\|_{X_0}^p dt\cr & =\|a\|_{B_{h,\C}^q}^p\int^\al_0  \|F_0\T_0(t)x\|_{X_0}^p dt\cr & \le \left(\ga \|a\|_{B_{h,\C}^q}\right)^p \|x\|_{X_0}^p,
\end{align*}
for any $x\in D(\A_0).$
This shows that $\Upsilon_0$ is $p$-admissible observation for $\A_0$. On the other hand, Lemma \ref{lemma:bergman_estimate} implies that the operator $\delta_0:D\left(\frac{d}{dz}\right)\subset B_{h,X_0}^q\to X$ is $p$-admissible for  $\frac{d}{dz}$. Using the above expression of the semigroup $(\T(t))_{t\ge 0}$ it is clear that $P$ is a $p$--admissible observation operator for $A$.
\end{proof}
Now we state the main result of this section:
\begin{thm}\label{good-thm}
Let $X_0$ be a UMD space,  $s>1$,  and $h$ be an admissible function  satisfying
$$\int_0^1 h(\si)^{1-s}d\si<+\infty.$$ Let $p>1$ and set $q = \frac{ps}{s-1}$. Assume that $a(\cdot)\in B_{h,\C}^q$	and $F_0\in\calL(D(\A_0),X_0)$ is a $p$-admissible observation operator for $\A_0$. If $\A_0$ has maximal $L^p$-regularity on $X_0$, then $\mathfrak{A}$ has the  maximal $L^p$-regularity on $X^q$.
\end{thm}
\begin{proof}
	In \cite{Barta2}, the author showed that $(\frac{d}{dz},D(\frac{d}{dz}))$ has the maximal $L^p$-regularity on $B_{h,X_0}^q$. By assumptions,  $\A_0$ has maximal $L^p$-regularity on $X_0$, then it is easy to see that $A$ has maximal $L^p$-regularity on $X^q$. Since $P$ is $p$-admissible for $A$ and $\mathfrak{A} =A + P$, Theorem \ref{theorem:Miyadera} guaranties that $\mathfrak{A}$ has the maximal $L^p$-regularity on $X^q$.
\end{proof}
\begin{rem}
With the notation in \cite[Thm.3.3]{Barta2}, we have $B(\cdot)=a(\cdot)F$, but  with the concept of the admissibility of the operator $F$ and according to Lemma \ref{lemma:observation_fractional_power} and Remark \ref{comparaison-Luts-small}, $B(\cdot)$ becomes small with respect to $\A_0$ in the sense of (i) in [5, Thm 3.3]. Thus the result of \cite{Barta2} can be obtained for this class of Volterra equation. Further, the author in \cite{Barta2},  proved  that condition (ii) of \cite[Thm.3.3]{Barta2} is also sufficient to obtain the required result. But its not clear for us how the author in \cite[Thm.3.3]{Barta2} can use \cite[Corollary 12]{KunstmannWeis} to obtain the result. However, under the condition (ii), the author in \cite[Thm.3.3]{Barta2} can get the desired result by observing only that (ii) implies easily(i). Our result of Maximal $L^p$-regularity of free-boundary Volterra equation (5.1) is not very strong, since the author  in \cite{Barta2}  shows it for a class of the operator $F$ that are not necessarily admissible. However, we show that the next example cannot be examined by the work \cite[Thm.3.3]{Barta2}.
\end{rem}
\subsection{Boundary perturbation of an integro-differential equation:}
We are still working under the setting of the previous section and we are now studding the problem:
\begin{equation}
\label{BIP}
\begin{cases}
\dot{\varrho}(t) = A_m\varrho(t)+\displaystyle\int_0^t a(t-s)F \varrho(s)ds +f(t) , & \mbox{ } t\geq 0  \\
G\varrho(t) = K\varrho(t), & \mbox{} t\geq 0 \\
\varrho(0) = 0 & \mbox{}
\
\end{cases}
\end{equation}
where $A_m:Z_0\to X_0$ is a closed linear operator, $G,K:Z_0\to U_0$ and $F:Z_0\to X_0$ are linear operators. Moreover, we introduce the operator
\begin{align*}
\A:=A_m,\qquad D(\A):=\{x\in Z_0:Gx=Kx\}.
\end{align*}
The problem \eqref{BIP} can now be written as
\begin{equation}
\label{BIP1}
\begin{cases}
\dot{\varrho}(t) = \A\varrho(t)+\displaystyle\int_0^t a(t-s)F \varrho(s)ds +f(t) , & \mbox{ } t\geq 0, \\
\varrho(0) = 0, & \mbox{}.
\
\end{cases}
\end{equation}
 This integro-differential equation is similar to that investigated in the previous subsection. We then use the same notation of Bergman space and product spaces. On  $X= X_0\times B_{h,X_0}^q$, let us define  the matrix operator
\begin{align*}
\mathfrak{G}=\begin{pmatrix} \A& \delta_0\\ \Upsilon& \frac{d}{dz}\end{pmatrix},\quad D(\mathfrak{G}):=D(\A)\times D\left(\frac{d}{dz}\right),
\end{align*}
where $\Upsilon x=a(\cdot)F x$ for $x\in Z_0$. As discussed in the previous subsection the maximal $L^p$-regularity of the integro-differential equation \eqref{BIP1} is reduced to look for conditions for which the operator $\mathfrak{G}$ is a generator on $X^q$ and has the maximal $L^p$-regularity on $X^q$.

We introduce the following assumptions:
 \begin{enumerate}
                 \item[{\bf(A1)}] $G:Z_0\to U_0$ is surjective
                 \item [{\bf(A2)}] $\A_0:=A_m$ with domain $D(\A_0):=\ker(G)$ is a generator of a $C_0$--semigroup on $X_0$.
\end{enumerate}
As discussed in Section \ref{sec:2}, the assumptions {\bf(A1)} and {\bf(A2)} imply that the Dirichlet operator
\begin{align*}
\D_\la:=\left(G_{|\ker(\la-A_m)}\right)^{-1}\in\calL(U_0,X_0),\qquad \la\in\rho(\A_0),
\end{align*}
exists. Define then
\begin{align*}
&\B_0:=(\la-\A_{0,-1})\D_\la\in\calL(U_0,X_{0,-1}),\qquad \la\in\rho(\A_0),\cr & \K_0:=K_{|D(\A_0)}\in\calL(D(\A_0),U_0),\cr & F_0:=F_{|D(\A_0)}\in\calL(D(\A_0),X_0).
\end{align*}
We also need the following hypotheses
\begin{enumerate}
                 \item[{\bf(A3)}] the triple operator $(\A_0,\B_0,\K_0)$ generates a regular linear system on $X_0,U_0,U_0$ with the identity operator $I_{U_0}:U_0\to U_0$ as an admissible feedback.
                 \item [{\bf(A4)}]  the triple operator $(\A_0,\B_0,F_0)$ generates a regular linear system on $X_0,U_0,X_0$.
\end{enumerate}
The following result shows the generation property of the operator $(\mathfrak{G},D(\mathfrak{G}))$.
\begin{prop}\label{appl2-prop}
Let assumptions {\bf(A1)} to {\bf(A4)} be satisfied. Then the operator  $(\mathfrak{G},D(\mathfrak{G}))$ generates a strongly continuous semigroup on $X^q$.
\end{prop}
\begin{proof}
Assumptions {\bf(A1)} to {\bf(A3)} show that the operator $(\A,D(\A))$ generates a strongly continuous semigroup on $X_0$, see Theorem \ref{theorem:Hadd_Manzo_Rhandi}. On the other hand, if in addition we consider the assumption {\bf(A4)}, then as in the proof of Theorem \ref{theorem:generalisation_of_Hadd_Manzo_Rhandi} (i), one can see that the operator $F\in\calL(D(A),X_0)$ is a $p$-admissible observation of $\A$.  Hence the rest of the proof follows exactly  in the same way as in the previous subsection.
\end{proof}
The main result of this subsection is the following.
\begin{thm}\label{Volterra-boundary}
let $X_0,U_0$ be UMD spaces,  $s>1$ a real number, and $h$ be an admissible function  satisfying
$$\int_0^1 h(x)^{1-s}dx<+\infty.$$ Let $p>1$ and set $q = \frac{ps}{s-1}$. Suppose that $a(\cdot)\in B_{h,\C}^q$,  assumptions {\bf(A1)} to {\bf(A4)} be satisfied,  $\A_0$ generates a bounded analytic semigroup and  there exists $\om>\max\{\om_0(\A_0);\om_0(\A)\}$ such that the sets $\{s^\frac{1}{p} R(\om+is,\A_{0,-1})\B_0;s\neq 0\}$ and $\{s^\frac{1}{p'} \K_0 R(\om+is,\A_0);s\neq 0\}$ are $\mathcal{R}$-bounded with $\frac{1}{p}+\frac{1}{p'}=1$. If $\A_0\in \mathscr{MR}(0,T;X_0)$ then $\mathfrak{G} \in \mathscr{MR}(0,T;X^q)$.
\end{thm}
\begin{proof}
According to Theorem \ref{theorem:staffans-weiss-R-boundedness}, we have $\A\in \mathscr{MR}(0,T;X_0)$. On the other hand as we have mentioned in the proof of Proposition \ref{appl2-prop} the operator $F$ is a $p$-admissible observation operator for $\A$. We then follow the same technique as in the proof of Theorem \ref{good-thm} to conclude that $\mathfrak{G} \in \mathscr{MR}(0,T;X^q)$.
\end{proof}

\begin{rem}
When perturbing boundary conditions of the integro-differential  Volterra equation \eqref{IACP}, even if the operator $F$ is a small perturbation of $\mathbb{A}_0$ which in turns implies that $B(\cdot)=a(\cdot)F$ satisfies the condition (i) of \cite[Thm 3.3]{Barta2} with respect to $\mathbb{A}_0$,  $B$ do not  necessarily satisfies the condition (i) \cite[Thm 3.3]{Barta2} with respect to $\mathbb{A}$. Hence , the result of \cite[Thm.3.3]{Barta2} is not applicable. It is to be noted that if we assume that $\A_0$ generates an analytic semigroup on $X_0$ and $F$ is $p$-admissible observation operator for $\A_0$ (in particular it is small perturbation for $\A_0$, due to Remark \ref{comparaison-Luts-small}), then $F$ is $p$-admissible for $\A$ (of course under the conditions ${\bf(A3)}$ and ${\bf(A4)}$). So that $F$ can be considered as a small perturbation for $\A$ due to Remark \ref{comparaison-Luts-small}. Even with this one cannot use [5, Thm 3.3] to conclude that the integro-differential  Volterra equation \ref{BIP1} has the maximal $L^p$-regularity, because for Barta \cite{Barta2} it is not clear that $\A$ has the  maximal $L^p$-regularity on $X$.

%
%
%
\end{rem}

\subsection{Example of a IDE with perturbed boundary conditions:}

Let $r,q\in(1,\infty)$ such that $\frac{1}{r}+\frac{1}{q}=1$. In this section, we first deal with the following PDE:
\begin{equation}
\label{PDE}
\begin{cases}
\frac{\partial }{\partial t} w(t,s) = \frac{\partial^2 }{\partial s^2} w(t,s) + f(t) , & \mbox{ } t\in [0,T],\ s\in (0,1)  \\
\frac{\partial }{\partial s}w(t,1) = w(t,1)-w(t,0), & \mbox{}  \\
\frac{\partial }{\partial s}w(t,0) = 0 & \mbox{}  \\
w(0,s) = 0, & \mbox{}
\
\end{cases}
\end{equation}

Now let $X_0=L^r(0,1)$ and $A_mf:=f''$ with domain $Z=\{w\in W^{2,r}(0,1),w'(0)=0 \}$. \\
We define on $Z$ the operator:
\begin{align*}
G:Z&\to \mathbb{R} \\
f&\to f'(0) .
\end{align*}
We also define the unbounded operator $K:Z\to \mathbb{R}$ by $Kf:=f(1)-f(0)$. One can see that under this setting, the problem (\ref{PDE}) can be transformed to :
\begin{equation*}
  \begin{cases}
  \dot{x}(t) = A_m x(t) , & \mbox{ } 0\leq t \leq T,\quad x(0)=0 \\
  Gx(t) = Kx(t), & \mbox{  } 0\leq t \leq T\quad .
  \
  \end{cases}
\end{equation*}
We define the operator:
$$
\A_0=A_m,\qquad D(\A_0)=Ker G=\{w\in W^{2,r}(0,1),w'(0)=0\quad and\quad w'(1)=0 \}.
$$
It is well known that this operator has maximal regularity. We set:
$$
\mathbb{D}_\lambda := (G_{|ker(\lambda -A_m)})
$$
and
$$
\B_0:=(\lambda - \A_{0,_{-1}})\mathbb{D}_\lambda
$$
and
$$
\K_0 :=K_{|D(\A_0)}
$$
Now the main result is the following theorem
\begin{thm}
The operator $\A$ defined by :
$$
\A f = f'',\qquad D(\A):=\{w\in W^{2,r}(0,1),w'(0)=0\quad and\quad w'(1)=w(1)-w(0) \},
$$
has maximal $L^p$-regularity for every $p>1$ and the problem (\ref{PDE}) has unique solution $w\in W^{1,p}([0,T];X)\cap L^{p}([0,T];D(\A))$ such that:
$$
\Vert \dot{w} \Vert_{L^p([0,T];X)} + \Vert w \Vert_{L^p([0,T];X)} + \Vert w''\Vert_{L^p([0,T];X)} \leq C\Vert f \Vert_{L^p([0,T];X)},
$$
for some constant $C>0$.
\end{thm}
\begin{proof}

We will show that conditions of Theorem \ref{theorem:staffans-weiss-R-boundedness} are satisfied. According to \cite{ABE} Theorem 2.4, to show that the triple $(\A_0,\B_0,\K_0)$ generates a regular linear system, all we have to show is that there exist $\beta \in [0,1]$ and $\gamma \in (0,1]$ such that :

\begin{enumerate}
\item[(i)] $range(\mathbb{D}_\lambda) \subset F_{1-\beta}^{\A_0}$ (where $F_{\alpha}^{\A_0}$ is the Favard space of $\A_0$ of order $\alpha$)
\item[(ii)] $D[(\lambda - \A_0)^\gamma] \subset Z$
\item[(iii)] $\beta + \gamma <1$.
\end{enumerate}
Let prove these three propositions:

\item[(i)] By simple calculation, one can see that there exist $\lambda_0  > 0$ such that
$$
\sup_{\lambda>\lambda_0}\Vert \lambda^\frac{r+1}{2r} \mathbb{D}_\lambda \Vert <+\infty.
$$
This fact implies (see \cite{ABE}, Lemma A.1) that $range(\mathbb{D}_\lambda) \subset F_{\frac{r+1}{2r}}^{\A_0}$ for some $\lambda\in \rho(\A_0)$. We take $\beta = 1-\frac{r+1}{2r} = \frac{r-1}{2r}$.

\item[(ii)] According to \cite[Lemma A.2]{ABE}, we have to show that for $\alpha \in (0,1)$ and every $\rho\geq \rho_0>0$ we have
$$
\vert \K_0 f\vert \leq M(\rho^\alpha \Vert f\Vert_r + \rho^{\alpha-1}\Vert f''\Vert_r),\quad f\in D(\A_0)
$$
for some constant $M>0$. For $f\in D(\A_0)$ we know that
$$
f(1)= f(0) + \int_0^1 f'(s)ds,
$$
then we can easily show that :
$$
\vert f(1)-f(0)\vert \leq \Vert f'\Vert_r\quad .
$$
By \cite[Example III.2.2]{EngNag}, we have for $\epsilon>0$
$$
\Vert f'\Vert_r \leq \frac{9}{\epsilon}\Vert f\Vert_r + \epsilon \Vert f''\Vert_r\quad .
$$
By taking $\rho = \epsilon^{-3}$, $\alpha = \frac{1}{3}$ and $\gamma \in (\frac{1}{3},\frac{1}{r})$(we can always take $1<r<3$), we show the assertion.
\item[(iii)] Clearly we have $\beta + \gamma <1$

Hence, by \cite[Theorem 2.4]{ABE}, the operators $\B_0$ and $\K_0$ are $p$-admissible for $\frac{2r}{r+1}<p<\frac{1}{\gamma}$ and the triple $(\A_0,\B_0,\K_0)$ generates a regular system and the identity is an admissible feedback. For instance we can take $p=2$.

To show that the sets $\{s^\frac{1}{2} \K_0 R(is,\A_0);s\neq 0\}$ and $\{s^\frac{1}{2} R(is,A_{0,-1})\B_0;s\neq 0\}$ are $\mathcal{R}$-bounded, it is sufficient to show that $\K_0$ and $\B_0$ are $l$-admissible ($l$-admissibility is more general than admissibility, see \cite{HaaKun} for definitions), namely, it is sufficient to show that the sets $\{s^\frac{1}{2}\K_0 R(s+1,\A_0);s>0\}$ and $\{s^\frac{1}{2}R(s+1,\A_{0,-1})\B_0;s>0\}$ are $\mathcal{R}$-bounded (see \cite[page 514]{HaaKun}).

In order to prove that $\{s^\frac{1}{2}\K_0R(s+1,\A_0);s>0\}$ is $\mathcal{R}$-bounded, we follow the example in \cite[page 528]{HaaKun} and the technique used there to show the $\mathcal{R}$-boundedness of the above set, namely it suffices to find a space $\tilde{Z}_0$ such that $D(\A_0)\subset \tilde{Z_0}\subset X_0$ and $\K_0$ is bounded in $\Vert \cdot \Vert_{\tilde{Z}_0\to\mathbb{R}}$ and the set $\{s^\frac{1}{2}R(s+1,\A_0);s>0\}$ is $\mathcal{R}$-bounded in $\mathcal{L}(\tilde{Z}_0,X)$, this holds, by the same example, for $\tilde{Z}_0=W^{1,p}(0,1)$ since $\K_0$ is bounded in $\mathcal{L}(\tilde{Z}_0,\mathbb{R})$.

Now to show the $\mathcal{R}$-boundedness of $\{s^\frac{1}{2}R(s+1,A_{0,-1})\B_0;s>0\}$ we follow the same example, we show that the set $\{s^\frac{1}{2}\B^{*}_0 R(s+1,\A^{*}_0);s>0\}$ is $\mathcal{R}$-bounded. Let us first determine $\B^{*}_0$, we proceed again as in \cite{HaaKun}, multiplying the first equation in (\ref{PDE}) with a fixed $v\in C^\infty([0,1])$ and integrating by parts we obtain:
$$
<w'(t,\cdot),v>_{(0,1)} = <w(t,\cdot),v''>_{(0,1)} + Gw(t,\cdot)\delta_1 v,
$$
this means that $\B^{*}_0  = \delta_1$. By the same argument used to show that $\{s^\frac{1}{2}\K_0R(s+1,\A_0);s>0\}$ is $\mathcal{R}$-bounded, we show that $\{s^\frac{1}{2}\B^{*}_0 R(s+1,\A^{*}_0);s>0\}$ is $\mathcal{R}$-bounded.

All hypothesis of Theorem \ref{theorem:staffans-weiss-R-boundedness} are satisfied, and hence the operator $\A$ defined above
has the maximal $L^2$-regularity, thus maximal $L^p$-regularity for every $p>1$ and then problem (\ref{PDE}) has unique solution $w\in W^{1,p}([0,T];X_0)\cap L^{p}([0,T];D(\A))$ such that:
$$
\left\Vert \frac{\partial }{\partial t} w(\cdot,\cdot) \right\Vert_{L^p([0,T];X_0)} + \Vert w(\cdot,\cdot) \Vert_{L^p([0,T];X)} + \Vert \frac{\partial^2 }{\partial s^2}w(\cdot,\cdot) \Vert_{L^p([0,T];X_0)} \leq C\Vert f \Vert_{L^p([0,T];X_0)},
$$
for some constant $C>0$.

\end{proof}
Now, we can consider the following partial integro-differential equation (PIDE) governed by the heat equation

\begin{equation}
\label{PDEV}
\begin{cases}
\displaystyle\frac{\partial }{\partial t} w(t,s) = \frac{\partial^2 }{\partial s^2} w(t,s)+\displaystyle \int_0^t a(t-\tau)F u(\tau,s)d\tau + f(t) , & \mbox{ } t\in [0,T],\ s\in (0,1)  \\
\displaystyle\frac{\partial }{\partial s}w(t,1) = w(t,1)-w(t,0), & \mbox{}  \\
\displaystyle\frac{\partial }{\partial s}w(t,0) = 0, & \mbox{}  \\
w(0,s) = 0, & \mbox{}
\
\end{cases}
\end{equation}
 where $F:Z_0\to X_0$ is a linear operator and $a\in B^q_{h,\C}$ (see notation above). We will use the same notation as in the beginning of this subsection.
We then assume that $\A_0$ has the maximal $L^p$-regularity. By applying Theorem \ref{Volterra-boundary}, we can easily check maximal regularity of (\ref{PDEV}) if the triple $(\A_0,\B_0,F_{|D(\A_0)})$ generates a regular linear system. Let $F= (-\A_0)^\theta$ for some $\theta \in (0,\frac{1}{p})$. By Remark \ref{remark-for-example}, the triple $(\A_0,\B_0,(-\A_0)^\theta_{|D(\A_0)})$ generates a regular linear system. Now, all assumptions of Theorem \ref{Volterra-boundary} are verified, the equation \eqref{PDEV} has the maximal $L^p$--regularity.

\bibliographystyle{amsplain}

\end{document}